\documentclass[a4paper,10pt]{article}

\usepackage{enumerate}
\usepackage{amsmath}

\usepackage{amsthm}

\usepackage{amssymb}
\usepackage{latexsym}
\usepackage{tikz}
\usepackage{tikz-cd}
\usepackage{braids}
\usetikzlibrary{
  knots,
  hobby,
  decorations.pathreplacing,
  shapes.geometric,
  calc
}
\usepackage[title]{appendix}

\newtheorem{Theorem}{Theorem}

\newtheorem{Proposition}[Theorem]{Proposition}
\newtheorem{Lemma}[Theorem]{Lemma}
\newtheorem{Corollary}[Theorem]{Corollary}
\newtheorem{Claim}[Theorem]{Claim}

\theoremstyle{definition}
\newtheorem{Definition}{Definition}
\newtheorem{Remark}{Remark}

\def\S{{\mathbf S}}

\def\FF{{\mathbb F}}
\def\CC{{\mathbb C}}

\def\ZZ{{\mathbb Z}}
\def\QQ{{\mathbb Q}}

\def\RR{{\mathbb R}}

\def\Fp{{\overline{\mathbb F}_p}}

\def\sll{\mathfrak{sl}}
\def\sl2{\sll_2\CC}

\def\SLL{\operatorname{SL}}
\def\SL2{\SLL_2\CC}

\begin{document}

\title{Examples of character varieties in characteristic $p$ and ramification}

\author{Luisa Paoluzzi and Joan Porti\footnote{Partially supported by the 
Spanish Mineco through grant MTM2015--66165--P}} 
\date{\today}

\maketitle

\begin{abstract}
We study $\mathrm{SL}_2(\FF)$-character varieties of knots over algebraically 
closed fields $\FF$. We give a sufficient condition in terms of the double 
branched cover of a $2$-bridge knot (or, equivalently, of its Alexander 
polynomial) on the characteristic of $\FF$, an odd prime, for the 
$\mathrm{SL}_2(\FF)$-character variety to present ramification phenomena. 
Finally we provide several explicit computations of character varieties to 
illustrate the result, exhibiting also other types of ramification. 
\end{abstract}


\section{Introduction}

Character varieties have turned out to be powerful tools in the comprehension 
of $3$-manifolds in general and knots in particular. In the case of knots,
different invariants are related or connected to character varieties. This is,
for instance, the case of the A-polynomial \cite{CooperLongApolynomial}, and 
the algebraic and geometric properties of the excellent component of a 
hyperbolic knot \cite{HLMJKTR, MPL, PetersenReid}. In this paper we are 
interested in a far less explored type of invariants, namely the finite set of 
odd prime numbers $p$ for which the character variety ramifies in
characteristic $p$. 

Recall that for a finitely presented group $\Gamma$ its 
$\mathrm{SL}_2(\CC)$-character variety, noted $X(\Gamma)$, is an algebraic set 
parameterizing, roughly speaking, the conjugacy classes of representations of 
$\Gamma$ into $\mathrm{SL}_2(\CC)$. The algebraic set $X(\Gamma)$ is determined 
by a finite set of polynomial equations with integers coefficients. If $p$ is a 
prime number, one can consider the polynomial equations obtained from the 
previous ones by reducing their coefficients mod $p$. It follows from work by 
Gonz\'alez--Acu\~na and Montesinos--Amilibia \cite{AcunaMontesinos} that these 
new equations define the variety of characters of representations of $\Gamma$ 
into $\mathrm{SL}_2(\FF)$, where $\FF$ is an algebraically closed field of 
characteristic $p$, provided that $p$ is odd.   

It is not hard to see that for almost every $p$ several features of the 
structure of the variety of characters of $\Gamma$ over $\FF$ coincide with
those of $X(\Gamma)$ over $\CC$: this is for instance the case of the number of
irreducible components of the algebraic set and their dimensions. We shall say
that \emph{$p$ ramifies} if there is a discrepancy between the behaviour of 
$X(\Gamma)$ and that of the character variety in characteristic $p$. We are
interested in understanding what types of ramification phenomena can appear in
character varieties of knots (i.e.~when $\Gamma$ is the fundamental group of a
knot exterior) and which primes ramify. 

The existence of two different types of ramification phenomena were pointed out
in previous papers by the authors. Both variations in the dimensions of
irreducible components and in their number can occur: the former phenomenon
appears for character varieties of orbifold structures of certain Montesinos
knots \cite{PP}, and the latter for knots admitting symmetries of order $p$ 
\cite{PP1}. In both cases, the appearance of ramification phenomena is related 
to the elementary fact that a matrix in $\mathrm{SL}_2(\FF)$ has order $p$ if 
and only if it is parabolic (or unitary), i.e.~has trace equal to $2$.

The present paper is devoted to exhibiting examples of yet another type of 
ramification phenomena that we might describe as \emph{order-$1$ ramification} 
in contrast to the ramification observed so far that we might call of 
\emph{order $0$}. 

To make this discussion more concrete, let us consider an explicit example: the
character variety of the figure-eight knot. It is well-known that this 
algebraic set can be defined by the equation $(x-2)(x^2+(1-t^2)x+t^2-1)$ where 
$t$ represents the trace of a meridian, while $x$ is the trace of the product 
of two generators of the group: a meridian and a conjugate of its inverse (see 
for instance \cite[Section~6]{AcunaMontesinos}, with the variables $z=t^2-x$ 
and $y=t^2$). The character variety consists of two $1$-dimensional irreducible 
components: the first one is the component of characters of abelian 
representations, while the second one is Thurston's excellent component 
containing the character of the lift to $\mathrm{SL}_2(\CC)$ of the hyperbolic 
holonomy of the knot. It is easy to convince oneself that, regardless of $p$, 
the variety of $\mathrm{SL}_2(\FF)$ always consists of two $1$-dimensional 
irreducible components. However, although both components are smooth and 
intersect in precisely two points over $\CC$, this is no longer the case for 
all $p$. Indeed, by computing the partial derivatives of the second component 
one has $2x-t^2+1$ and $-2t(x-1)$ and one realises that, for $p=5$, the second 
component is no more smooth at the point $t=0$ and $x=2$: the point is a 
cuspidal point and the two points of intersection between the two components 
collapse to this singular one in characteristic $5$. Observe that if one 
considers $\mathrm{PSL}_2$-characters instead of $\mathrm{SL}_2$-characters, 
that is if we take $T=t^2$ as a coordinate instead of $t$, then the point of 
coordinates $T=0$ and $x=2$ is smooth even in characteristic $5$ on each of the 
components, but the intersection between both components is transverse in 
characteristic zero or $p\neq 5$, and tangent in characteristic $p=5$.

\begin{Definition}
We say that $X(\Gamma)$ has a ramification of \emph{order $1$ type in
characteristic $p$} if in the character variety in characteristic $p$ either 
{ the type of intersection beteen two irreducible components 
changes (in particular, two transverse components become tangent)}, 
or the singularity type of a 
point changes (in particular, a singular point appears on a smooth component).
\end{Definition}

Note that ramifications of order $0$ and $1$ may appear at the same time if,
for instance, an irreducible component splits into two and they meet at a
point.

The following result gives a conceptual explanation of the reason why $5$ is a
ramified prime for the character variety of the figure-eight knot. 

\begin{Theorem}\label{th:bridge}
Let $K$ be a $2$-bridge knot, $\Delta_K(t)$ its Alexander polynomial, and $p$ a
prime. We consider $X(\pi_1(K))$. If $p$ divides $|\Delta_K(-1)|$, that is the
order of the homology of the $2$-fold branched cover of $K$, then $p$ ramifies.
\end{Theorem}

The above is a special case of more technical result for whose statement we 
need to introduce some notation first. Let $K$ be a knot and $M_n$ its $n$-fold 
cyclic branched cover. Let $\mu\in\pi_1(K)$ be an element representing a 
meridian of $K$. We have the following exact sequence of groups
$$1\longrightarrow H\longrightarrow \pi_1(K)/<<\mu^4>>\longrightarrow
\ZZ/2\ZZ\longrightarrow 1.$$
where $H$ can be seen as the orbifold fundamental group of the orbifold whose
underlying topological space is $M_2$ and whose singular locus is the lift of
$K$ with order $2$ singularity. The commutator subgroup $H'$ of $H$ is 
characteristic, so we can consider the quotient 
$\Gamma=(\pi_1(K)/<<\mu^4>>)/H'$. 

\begin{Theorem}\label{th:main}
Let $K$ be a knot, $\Delta_K(t)$ its Alexander polynomial, and $p$ a prime. We
consider $X(\Gamma)$, where $\Gamma$ is defined as above. If $p$ divides 
$|\Delta_K(-1)|$, that is the order of $H_1(M_2)$, then $p$ ramifies for 
$X(\Gamma)$.
\end{Theorem}

Remark that, for $2$-bridge knots, $\Gamma$ is a binary dihedral group,
extension of $\ZZ/2\ZZ$ by $\pi_1(M_2)=H_1(M_2)$. In addition, 
$\mathrm{SL}_2(\CC)$-representations of $\Gamma$ are lifts of representations 
of $\pi_1(K)/<<\mu^2>>$ to $\mathrm{PSL}_2(\CC)$.

A natural question is to determine what ramification phenomena can appear in 
the situation of Theorem~\ref{th:main} and if the ramification is for $X(K)$ 
itself and not just for its algebraic subset $X(\Gamma)$. The easiest class of 
examples to consider is that of torus knots (see Section~\ref{sec:torus}). Over 
$\CC$ the character variety of a torus knot consists of a finite number of 
rational curves and the only ramification phenomena that occur are when several 
irreducible components collapse onto a single one. It follows that the
character variety of the torus knot $T(a,b)$ ramifies for every odd prime $p$
dividing $ab$, in particular whenever $p$ divides $|\Delta_{T(a,b)}(-1)|$.  

Although torus knots show that the ramification phenomena that can be observed 
in the situation of Theorem~\ref{th:main} is not always of the type seen for 
the figure-eight knot, one might expect that this latter behaviour is generic, 
at least for hyperbolic knots. The following corroborates this hypothesis.

\begin{Proposition}\label{p:ram1}
Let $K$ be a hyperbolic $2$-bridge knot with Alexander polynomial $\Delta_K(t)$.
Let $p$ be a divisor of $|\Delta_K(-1)|$. If all roots of $\Delta_K$ are 
simple, then the ramification phenomena occurring at $p$ is of order $1$.  
\end{Proposition}

By \cite{HPS} the hypothesis that all roots of $\Delta_K(t)$ are simple implies 
that the curve of abelian characters intersects transversely the other 
components of $X(K)$. In Lemma~\ref{lemma:2bridgeatt0} we show that the 
intersection is no more transverse but tangent in characteristic $p$ whenever 
$p$ divides $|\Delta_K(-1)|$. Furthermore, the arguments of \cite{HPS} can be 
used to argue that ramifications in the intersection of both components will 
appear at most for primes that divide the discriminant of $\Delta_K(t)$, 
provided this discriminant does not vanish (namely, all roots of $\Delta_K(t)$ 
are simple). See Remark~\ref{Remark:Discriminant}. Notice that $\Delta_K(-1)$ 
divides the discriminant of $\Delta_K(t)$. 


Incidentally, the techniques of the proof yield the following fact which does 
not seem to be stressed elsewhere in the literature.


\begin{Proposition}\label{p:irred}
Let $K$ be a hyperbolic $2$-bridge knot with Alexander polynomial $\Delta_K(t)$.
Let $\alpha=|\Delta_K(-1)|$. If $\alpha$ and $(\alpha-1)/2$ are prime numbers, 
then the character variety of $K$ over $\CC$ consists precisely of two 
irreducible components: the one whose characters correspond to abelian 
representations and Thurston's excellent one.
\end{Proposition}

At this point we do not know whether the condition $(\alpha-1)/2$ prime in the
above proposition is in fact necessary. 
On the other hand, there are examples of hyperbolic $2$-bridge knots such that 
neither $\alpha$ nor $(\alpha-1)/2$ are prime numbers and yet their character 
varieties consist of just two irreducible components.

To complete our analysis and to improve our understanding of ramification
phenomena, we provide explicit computations of character varieties for some
knots, notably for the tunnel number-one $P(2,2,3)$-pretzel knot $8_5$ and for
the $\pi$-hyperbolic knot $8_{18}$. We also study character varieties for some 
$2$-bridge knots. Among them, the knot $8_9$ admits a ramified prime that does 
not come from the situation of Theorem~\ref{th:main}. Quite interestingly, in
this case we observe a new type of ramification, namely an irreducible
component splitting into two irreducible ones. 

Note that the aforementioned examples cover all possible geometries of double 
branched covers of hyperbolic knots, that is spherical, Seifert fibred, and 
hyperbolic. As for the case of hyperbolic knots with double branched covers 
admitting a non trivial JSJ-decomposition, their character varieties seem hard 
to compute because of the ``large number" of generators of their groups (none 
of these is a tunnel number-one knot, for instance). Instead, we compute 
explicitly the character varieties of two tunnel number-one satellite knots, 
both with companion the trefoil knot. The first is a cable knot whose double 
branched cover is a graph manifold, while the second is a Whitehead double so 
the JSJ-decomposition of its double branched cover contains a hyperbolic piece. 
As a side remark, the double branched covers of these two knots are also the 
double branched covers of hyperbolic knots. Of course, it is not clear if 
there might be a relationship between the character varieties of two knots 
sharing the same double cover.

The paper is organised as follows. { In section~\ref{s:chvar} we 
recall basic facts about character varieties. The proofs of 
Theorems~\ref{th:bridge} and \ref{th:main} follow from the discussion in
Section~\ref{sec:t=0}, while Section~\ref{sec:twobridge} deals with the 
specificities of $2$-bridge knots. This latter section also contains detailed
computations of character varieties for some $2$-bridge knots and the primes for
which they ramify. The character varieties of the knots $8_5$ and $8_{18}$ are 
studied in Sections~\ref{sec:pretzel} and \ref{sec:pihyperbolic}, respectively. 
Section~\ref{sec:torus} is dedicted to character varieties of torus knots,
while in Section~\ref{sec:sat} we discuss two examples of character varieties
of satellite knots. Finally, some auxiliary computations are provided in the 
Appendices.}

\medskip

The computations in this paper were performed by using Maple\texttrademark.

\paragraph*{Acknowledgements} The authors are indebted to Alan Reid for 
suggesting the exploration of order 1 type ramification. L. Paoluzzi is also
thankful to UAB for hospitality during her visit when the contents of the 
present paper were originally discussed.

\color{black}


\section{Varieties of characters}\label{s:chvar}

In this section we recall relevant facts about character varieties and 
introduce some notation.

Let $\Gamma$ be a finitely presented group and $\FF$ an algebraically closed
field. Given a representation $\rho \colon \Gamma \longrightarrow 
\mathrm{SL}_2(\FF)$, one can define a map $\chi_\rho \colon \Gamma 
\longrightarrow \FF$ by composing with the trace function. Such map is called 
\emph{the character of the representation $\rho$}. We will denote 
$X(\Gamma)_\FF$ the set of all characters of representations of $\Gamma$. Note 
that each element $\gamma$ of $\Gamma$ induces an evaluation map 
$\tau_\gamma \colon X(\Gamma)_\FF \longrightarrow \FF$ defined as 
$\tau_\gamma(\chi)=\chi(\gamma)$.

For simplicity, when $\FF=\CC$ we will usually write $X(\Gamma)$ instead of 
$X(\Gamma)_\CC$. 

\begin{Proposition}[\cite{CullerShalen,AcunaMontesinos}]
\label{proposition:X(G)}
The set of characters $X(\Gamma)$ is an affine algebraic set defined over
$\ZZ$, which embeds in $\CC^N$ with coordinate functions 
$(\tau_{\gamma_1}, \ldots, \tau_{\gamma_N})$ for some
$\gamma_1,\ldots,\gamma_N\in \Gamma$.
\end{Proposition}

The affine algebraic set $X(\Gamma)$ is called the \emph{character variety} of 
$\Gamma$: it can be interpreted as the algebraic quotient of the variety of
representation of $\Gamma$ into $\mathrm{SL}_2(\CC)$ by the conjugacy
action of $\mathrm{PSL}_2(\CC)=
\mathrm{SL}_2(\CC)/\mathcal{Z}(\mathrm{SL}_2(\CC))$.

Note that the set $\{\gamma_1,\ldots,\gamma_N\}$ in the above proposition can
be chosen to contain a generating set of $\Gamma$. For $\Gamma$ the fundamental 
group of a knot exterior, we will then assume that it always contains a 
representative of the meridian.

It is worth pointing out that the polynomial equations with integer
coefficients of Proposition~\ref{proposition:X(G)} only determine $X(\Gamma)$
as an algebraic set and not as a scheme. 

A careful analysis of the arguments in \cite{AcunaMontesinos} shows that
Proposition~\ref{proposition:X(G)} still holds if $\CC$ is replaced by any 
algebraically closed field $\FF$, provided that its characteristic is different
from $2$. Let $\FF_p$ denote the field with $p$ elements and $\Fp$ its
algebraic closure. We have:

\begin{Proposition}[\cite{AcunaMontesinos}]
\label{proposition:X(G)Fp}
Let $p>2$ be an odd prime number. The set of characters $X(\Gamma)_{\Fp}$ 
associated to representations of $\Gamma$ into $\mathrm{SL}_2(\Fp)$ is an 
algebraic set which embeds in $\Fp ^N$ with coordinate functions 
$(\tau_{\gamma_1}, \ldots, \tau_{\gamma_N})$ for the same 
$\gamma_1,\ldots,\gamma_N\in \Gamma$ seen in 
Proposition~\ref{proposition:X(G)}. Moreover, $X(\Gamma)_\Fp$ is defined by the 
reductions mod $p$ of the polynomials with coefficients in $\ZZ$ which define 
$X(\Gamma)$.
\end{Proposition}

The key observation here is that, as in the case of complex numbers, the 
algebraic set $X(\Gamma)_{\Fp}$ can be again interpreted as the algebraic 
quotient of the variety of representation of $\Gamma$ into $\mathrm{SL}_2(\Fp)$ 
by the conjugacy action of $\mathrm{PSL}_2(\Fp)=
\mathrm{SL}_2(\Fp)/\mathcal{Z}(\mathrm{SL}_2(\Fp))$. 

We will be mainly interested in the situation where $\Gamma$ is the fundamental
group $\pi_1(\mathbf{S}^3\setminus K)$ of the exterior of a knot $K$. In this 
case we shall write $X(K)_\FF$ or $X(K)$ instead of 
$X(\pi_1(\mathbf{S}^3\setminus K))_\FF$ or $X(\pi_1(\mathbf{S}^3\setminus K))$ 
respectively.

Since the abelianisation of a knot group is infinite cyclic, for each element 
$A$ of $\mathrm{SL}_2(\FF)$ there is precisely one abelian representation of 
$\pi_1(\mathbf{S}^3\setminus K)$ defined by sending a meridian of $K$ to $A$. 
The set of abelian representations projects onto a rational, one dimensional 
component, of $X(K)_\FF$.  

If $K$ is a hyperbolic knot, $\pi_1(\mathbf{S}^3\setminus K)$ admits a 
hyperbolic holonomy representation into $\mathrm{PSL}_2(\CC)$. It is not hard 
to see that such representation lifts to a representation 
$\pi_1(\mathbf{S}^3\setminus K) \longrightarrow \mathrm{SL}_2(\CC)$. It was 
proved by Thurston that the character of this representation is a smooth point 
on a one-dimensional irreducible component of $X(K)$, known as the 
\emph{canonical} or \emph{excellent component} of $K$. More generally, Thurston 
proved that every irreducible component of $X(K)$ has dimension at least one
\cite{KapovichBook, ThurstonNotes}.




\section{Characters vanishing at the meridian}
\label{sec:t=0}

Consider the map $$t=\tau_{\mu}\colon X(K)\to \CC,$$ 
sending a character to the trace of the meridian. We are interested in its 
fiber at $0$:
$$
Y:= t^{-1}(0)\subset X(K).
$$
Note that if we choose the trace of the meridian to be one of the coordinates
of $X(K)$, the map is just a projection onto a coordinate and the fibre the
intersection of $X(K)$ with a hyperplane. The condition $t(\chi_\rho)=0$ means 
that $\chi_\rho$ is the character of a representation $\rho$ that maps $\mu$ to 
a matrix $A$ such that $A^2=-\mathrm{Id}$. Consider the induced representation 
$\hat\rho$ obtained by composing $\rho$ with the natural quotient 
$\mathrm{SL}_2(\CC)\longrightarrow \operatorname{PSL}_2(\mathbb{C})$~: 
$\hat\rho$ restricts to a representation $\hat\rho'$ of the double branched 
covering, $M_2$, of $K$.

We decompose the set $Y=t^{-1}(0)$ into three (possibly empty) disjoint 
subsets, according to 
the behaviour of $\hat\rho'$ 
$$
Y =Y^{triv}(K) \sqcup Y^{ab}(K) \sqcup Y^{nab}(K),
$$
where: 
\begin{itemize}
        \item $\chi\in Y^{triv}(K)$ if there is a representation $\rho$ with
character $\chi$ such that $\hat\rho'$ is trivial; 
%
        \item  $\chi\in Y^{ab}(K)$ if there is a representation $\rho$ with
character $\chi$ such that $\hat\rho'$ is abelian, but no such abelian 
representation is trivial;
         \item  $Y^{nab}(K)$ if every representation $\rho$ with character 
$\chi$ is such that $\hat\rho'$ is not abelian.
       \end{itemize}

Note that $Y^{triv}(K) = X^{ab}(K) \cap Y$ always consists of a single point, 
the character of the abelian representation in $ t^{-1}(0) $.

In characteristic zero, $Y^{ab}(K) $ is the set of binary dihedral 
representations and also the set of metabelian (and non-abelian) representations
(see \cite{Lin,Nagasato}). As the cardinality of $H_1(M_2,\ZZ)$ is 
$\vert \Delta_K(-1)\vert$, an odd integer, the cardinality of $Y^{ab}(K)$ is 
$$\vert Y^{ab}(K) \vert = \frac{\vert \Delta_K(-1)\vert-1}{2}.$$ 
Thus $Y^{ab}(K)=\emptyset$ if and only if $\Delta_K(-1)=\pm 1$.

Note that the subset $Y^{triv}(K)\cup Y^{ab}(K)$ is a subvariety of $Y(K)$ and
thus of $X(K)$. Indeed, this set corresponds to the fibre $ t^{-1}(0) $ of
$X(\Gamma)$, where $\Gamma$ is the quotient of $\pi_1(\mathbf{S}^3\setminus K)$ defined in the introduction.


Before describing the reduction mod $p$ of $Y^{ab}(K) $ we need a lemma on 
finite order elements of $\mathrm{PSL}_2(\mathbb K)$. Its proof is a
straightforward consequence of the structure of $\mathrm{PSL}_2(\FF_{p^k})$
(see \cite{Dickson}).

\begin{Lemma} 
Let $\FF$ be an algebraically closed field of characteristic $p\geq 3$ 
and $A\in\mathrm{PSL}_2(\FF)$ an element of finite order. Then either 
$\operatorname{order}(A)$ is coprime with $p$ or $\operatorname{order}(A)=p$. 
In addition:
\begin{itemize}
\item When $\operatorname{order}(A)$ is coprime with $p$, then $A$ is 
diagonalisable.
\item When  $\operatorname{order}(A)=p$, then $A$ is conjugate to a matrix of 
the form
$\pm \left(  \begin{smallmatrix}
1 & * \\ 0 & 1
\end{smallmatrix}\right)$ 
\end{itemize}
\end{Lemma}

\begin{Corollary}
For any prime $p$, write $\Delta_K(-1)=\pm p^k m$, with $m\geq 1$ an integer 
coprime with $p$. Then the reduction mod $p$ of $Y^{ab}(K)$ has cardinality 
$\frac{m-1}{2}$.
\end{Corollary}

The previous corollary is an example of collapse when reducing mod $p$. Notice 
that when $m=1$, i.e.~when $\Delta_K(-1)=\pm p^k$, the reduction mod $p$ of 
$Y^{ab}(K)$ collapses to a single point, the abelian character in 
$Y^{triv}(K)$. Note that, in characteristic $p$, this is now also the character
of a reducible dihedral representation.

\medskip

Regarding $Y^{nab}(K)$, we are interested in the following cases for $K$:
\begin{enumerate}
\item $K$ a two bridge knot. Then $Y^{nab}(K)=\emptyset$ because $M_2$ is a 
lens space.
\item $K$ is a Montesinos knot. Then, as $M_2$ is Seifert fibred, $Y^{nab}(K)$ 
consists of representations whose projections into 
$\operatorname{PSL}_2(\mathbb{C})$ factor trough the base $2$-orbifold of the 
Seifert fibration. In particular for pretzel knots $Y^{nab}(K)$ is finite, 
because the variety of characters of triangle groups has finite cardinality.
\item For $\pi$-hyperbolic knots $ Y^{nab}(K)$ contains the lift of the 
holonomy of the orbifold with cone angle $\pi$. This is an isolated point of 
$Y(K)$ (Weil local rigidity), but there may be more characters in $Y^{nab}(K)$.
\end{enumerate}


\section{Two-bridge knots}
\label{sec:twobridge}

Let $K$ be a $2$-bridge knot. Recall that $K$ is determined by its 
\emph{$2$-bridge notation}, a rational number of the form $\beta/\alpha$, with 
$\alpha$ odd: one can recover a four-plat description of $K$ from a continued 
fraction expansion for $\beta/\alpha$, and the resulting knot does not depend 
on the chosen expansion. Moreover the $2$-fold branched cover $M_2$ of $K$ is 
precisely the lens space $L(\alpha,\beta)$ (see \cite[Ch 12]{BZ}). 


The fundamental group $\pi_1(\mathbf{S}^3\setminus K)$ of a $2$-bridge knot $K$ 
admits a presentation of the form $\langle a, b\mid a w= w b \rangle$, with $a$ 
and $b$ meridians and $w$ a word in the generators, $a$ and $b$, and their 
inverses.

Since $a$ and $b$ are conjugate, the character variety $X(K)$ of $K$ is a 
plane curve with coordinates $t= \tau_a=\tau_b$ and $x=\tau_{ab^{-1}}$:
$$
\{(t,x)\in\CC^2\mid P(t, x)(x-2)=0\},
$$
where $x=2$ is the set of characters of abelian representations, and its 
complement consists of characters of irreducible representations.

{
Notice that $P$ is an even polynomial of the variable $t$ (i.e.~a polynomial in 
$t^2$): there is a natural action of $H^1(S^3- K;\mathbb Z/2)\cong \mathbb Z/2$ 
on the variety of characters, that maps a point with coordinates $(t,x) $ to 
$(-t,x)$. By taking $T=t^2$ as a new variable, this defines the variety of 
$\mathrm{PSL}_2( \CC)$-characters \cite{AcunaMontesinos, HLMJKTR}.

Because of their simple structure, character varieties of $2$-bridge knots have
been widely studied, see for instance \cite{HLMJKTR, ORS, Oh, MPL, Chu},
see also \cite{HLMlinks} for $2$-bridge links.
}

Here, again, we are interested in $Y(K)$. We have:
$$
P(0,x)=\Phi_\alpha(x),\qquad \textrm{ where } 
\Phi_\alpha(\lambda+\frac{1}{\lambda})=
\frac{\lambda^\alpha-1}{\lambda-1} \lambda ^{\frac{1-\alpha}{2}}.
$$
Namely, the factors of $\Phi_\alpha$ define the intersection with $\RR$ of the 
corresponding cyclotomic extension.

\begin{Lemma}\label{lemma:2bridgeatt0}
For $p\mid \alpha$, write $\alpha= p^r\alpha'$, with $\alpha'\in \ZZ$ coprime 
with $p$.
\begin{enumerate}
\item $P(0,x)=0$ consists of $\frac{\alpha-1}{2}$ smooth points in $\mathbb{C}$.
\item The reduction mod $p$ of $P(0,x)=0$ consists of $\frac{\alpha'-1}{2}$ 
points in $\overline{\mathbb{F}_p}$. 
\item If $p^r\neq 3$, then the intersection between $P(t,x)=0$ and $x=2$ mod 
$p$ is not transverse.
\end{enumerate}
\end{Lemma}

\begin{proof}
The assertions on the cardinality are proved in the previous corollary. 
Smoothness is a consequence from the fact that the roots of $\Phi_\alpha$
are simple. It is also proved in \cite[Lemmas~4.4 and~4.5]{BoyerZhangJAMS}.

In terms of the polynomial $\Phi_\alpha$ we have:
$$
\Phi_\alpha(x)\equiv (x-2)^{\frac{p^r-1}{2}} (\Phi_\alpha(x))^{p^r}\mod p
$$

For assertion 3) use that
$$
P(t,x)=\Phi_ \alpha(x) + t^2 Q(t^2,x)
$$
and the previous congruence of $\Phi_\alpha(x)$ mod $p$.
\end{proof}

\begin{Remark}
{
One may ask whether this order-1 ramification phenomenon is a consequence of an
order-0 ramification phenomenon that occurs at the level of the character
variety $X(K)$, namely that different irreducible components of 
$\{(t,x)\in\CC^2\mid P(t, x)=0\}$ get identified together. 

More specifically, for $\alpha=p$ we may ask whether $P(t,x)$ can be a power 
mod $p$. As the characteristic $p$ is odd, if $P(t,x)$ were a power mod $p$, 
then it would be the power of an even polynomial in $t$. However in 
\cite[Proposition~5.2]{HLMJKTR} it is shown that the degree of $P(t,x)$ in 
$t^2$ is strictly less than the degree of $P(0,x)$, so this can never be the
case.
}
\end{Remark}

\begin{Remark}
\label{Remark:Discriminant}
Let $D$ be discriminant of $\Delta_K(t)$. Assume that $D\neq 0$. Then, in 
characteristic zero, all roots of $\Delta_K(t)$ are simple, and by \cite{HPS}  
the intersection of the abelian component $x=2$ with $P(t,x)=0$ is everywhere 
transverse. In characteristic $p$, if $p\nmid D$, then the intersection is 
still transverse, in particular  the points $(t,2)$ are smooth points of
$P(t,x)=0$. When $p\mid D$, new singularities or tangencies with $x=2$ may 
appear at those points. Notice also that $\Delta_K(-1)\mid D$.
\end{Remark}


In some instances, we may also count how many components $X(K)$ has:

\begin{Lemma} 
\label{lemma:irreducible}
If $\alpha$ is prime, then $P(t,x)$ is irreducible over $\mathbb Q$. 
If in addition $\frac{\alpha-1}2$ is also prime, then it is irreducible over 
$\mathbb C$. 
\end{Lemma}

\begin{proof}
The curve of characters $X(K)$ has no ideal points with $t=0$, by 
\cite[Lemma~A.1.1]{BoileauPorti}, hence every irreducible component of $X(K)$ 
intersects the line $t=0$ in at least one point. Thus, to prove that $P$ is 
irreducible (over $\QQ$ or over $\CC$) we need to check that all points in 
$$Y^{ab}(K)=\{(0, x)\mid P(0, x)=0\}$$
are contained in the same irreducible component (recall that here 
$Y^{nab}(K)=\emptyset$). As $P(0,x)=\Phi_\alpha(x)$ is irreducible over $\QQ$, 
because $\alpha$ is prime, $P=0$ has only one component over $\QQ$, so $P$ is 
irreducible over $\QQ$.

For the second assertion, we notice first that at least two points of 
$Y^{ab}(K)$ lie in the same component over $\CC$. Those points correspond to 
the lifts of the spherical holonomy of the orbifold with cone angle $\pi$. More 
precisely, the universal covering of $SO(4)$ is $SU(2)\times SU(2)$, and the 
lift of the holonomy yields two different characters in $Y^{ab}(K)$, one for 
the projection to each factor $SU(2)$, see for instance \cite{HLMTokyo}. 
In addition, there is a curve of smooth points that joins both characters 
\cite{BoileauPorti,Weiss, PortiWeiss, PortiKobe}, coming from holonomies of 
spherical and Euclidean cone structures, hence they lie in the same 
irreducible component.

Once we know that two points of $Y^{ab}(K)$ lie in the same component, we 
consider the action of the Galois group. If $P$ decomposes in $\CC$, it 
decomposes with coefficients in a number field  $\mathbb K$, that we may 
assume to contain the real cyclotomic field (containing the solutions of 
$\Phi_\alpha(z)$). The Galois group $\operatorname{Gal}(\mathbb K\mid\QQ)$ acts 
on $P=0$ by preserving or permuting irreducible components, in particular if
$P=0$ contains more than one irreducible component, then $Y^{ab}(K)$ has a 
nontrivial partition  induced by the components of $P=0$. As 
$\operatorname{Gal}(\mathbb K\mid\QQ)$ acts transitively on $Y^{ab}(K)$, all 
subsets of this partition have the same cardinality, which must be either $1$ 
or $(\alpha-1)/2$, because $(\alpha-1)/2$ is prime. It cannot be $1$ because we showed in the previous paragraph that the cardinality is at least $2$. Hence it 
is $(\alpha-1)/2$ and there is a single component.
\end{proof}

The reader should compare this result with \cite{Oh} where a condition is
provided for the character variety of a $2$-bridge knot to have at least two
irreducible components, besides the abelian one. 
 
\medskip

In the following we provide some examples of character varieties of $2$-bridge 
knots, for different values of $\alpha$, notably $\alpha$ a prime number, a
power of a prime number, and a product of different primes, in order to cover
potentially different kinds of behaviour. In the first three instances, we also
compute all ramified primes.

\begin{center}
  \textsc{Example: $6_1=7/9$ }
 \end{center}

From the presentation
$$
\pi_1(\mathbf{S}^3\setminus 6_1)=
\langle\alpha,\beta\mid \beta\alpha^{-1}\beta\alpha\beta^{-1}\alpha^{-1}\beta\alpha\beta^{-1}=
\alpha^{-1}\beta\alpha\beta^{-1}\alpha^{-1}\beta\alpha\beta^{-1}\alpha\rangle ,
$$
using coordinates
$$
t=\operatorname{tr}_\alpha=\operatorname{tr}_\beta,\qquad x=\operatorname{tr}_{\alpha\beta^{-1}},
$$
and taking traces of the terms in the relation, we easily compute
$P(t, x)$:
$$
{t}^{4}{x}^{2}-3\,{t}^{4}x-2\,{t}^{2}{x}^{3}+2\,{t}^{4}+2\,{t}^{2}{x}^
{2}+{x}^{4}+5\,{t}^{2}x+{x}^{3}-4\,{t}^{2}-3\,{x}^{2}-2\,x+1
$$

Even if $9$ is not prime, $P(t, x)$ is irreducible over $\mathbb{C}$, as we
will see in a while.
We give here a reason why it is irreducible over $\QQ$. We reproduce the 
argument in the proof of Lemma~\ref{lemma:irreducible}: setting $t=0$, we have
$$
P(0,x)=\Phi_{9}(x)= \Phi_3(x) \psi(x),
$$
where the solution of $\Phi_3(x)=x+1=0$ is $-1$, twice the real part of the primitive 
roots of unity of order $3$, and the solutions of $\psi(x)= x^3-3 x+1 =0$ are twice the 
real parts of primitive roots of unity of order $9$, see~\eqref{eqn:phi_25}. 
Thus, as there is no ideal point with $t=0$ \cite[Lemma~A.1.1]{BoileauPorti},
to prove irreducibly over $\QQ$ we must show that there are two points 
in the same component, with coordinates $(0,-1)$ and $(0,x_2)$, with
$\psi(x_2)=0$. Following the argument of
Lemma~\ref{lemma:irreducible}, those points are provided by the lift to
$\mathrm{SU}(2)\times \mathrm{SU}(2)$ of the holonomy of the spherical
orbifold. Up to conjugation, the lift of the holonomy of $\pi_1(M_2)$ is a
cyclic group generated by
$$
\left(
\begin{pmatrix}
 e^{i\,\theta_1/2} & 0 \\
 0 &  e^{i\,\theta_1/2}
\end{pmatrix}
, 
\begin{pmatrix}
 e^{i\,\theta_2/2} & 0 \\
 0 &  e^{i\,\theta_2/2}
\end{pmatrix}
\right)
$$
By the description of the action of $\mathrm{SU}(2)\times \mathrm{SU}(2)$ on
$\mathrm{SU}(2)\cong S^3$ \cite{Montesinos}, and as $M_2\cong L(9,7)$, 
$\theta_1$ and $\theta_2$ must satisfy
$$
\theta_2-\theta_1= \frac{1}{9}2\pi,\qquad \theta_1+\theta_2=\frac{7}{9}2\pi.
$$
Hence
$$
\theta_1=\frac{1}{3} 2\pi \qquad  \theta_2=\frac{2}{9} 2\pi ,
$$
and $x_1=2\cos\theta_1=-1$ and $x_2=2\cos\theta_2$ are the values we are looking
for.

Observe also that the action of the Galois group permutes the points $(0,x_2)$
with $x_2^3-3 x_2+1=0$, while fixing the point $(0,-1)$, so there is a single 
component over $\CC$.

As just observed, for $t=0$ we have:
$$
P(0,x)=\Phi_3(x) \psi(x) = \left( x+1 \right)  \left( {x}^{3}-3\,x+1 \right). 
$$
According to Theorem~\ref{th:bridge}, $Y(K)$ collapses to a single point 
mod 3, because
$$
P(0,x)\cong (x-2)^4 \mod 3.
$$
Note that the curve of irreducible characters meets the line of abelian ones in
two points with multiplicity $2$: these are computed as the solutions of the
equation $0=P(t,2)=-2t^2+9$. The two components also have an intersection at 
infinity, also of multiplicity $2$. The two (finite) intersections collapse to 
a single one in characteristic $p=3$ (and in this characteristic only), and the 
two components become tangent. The latter assertion is easily seen by rewriting 
$P(t,x)=t^4[(x-2)^2+(x-2)]+t^2[2(x-2)^3-10(x-2)^2+11(x-2)+42]+
(x+1)[(x-2)^3+6(x-2)^2+9(x-2)+3]$ so that one also observes that the component
is not smooth at this point in characteristic $p=3$.

We want to show that $p=3$ is the only ramified prime for the character variety
of the knot $6_1$. Clearly, for no prime the dimension of the variety can
decrease or increase and it will always consists of at least two irreducible
components. We will prove that the component defined by $P(t,x)=0$ is smooth
over $\CC$, and so is the corresponding component in any characteristic $p>3$. 

The partial derivatives are 
$$\partial_t P(t,x)=2t[t^2(2x^2-6x+4)-2x^3+2x^2+5x-4]$$
and
$$\partial_x P(t,x)= t^4(2x-3)+t^2(-6x^2+4x+5)+4x^3+3x^2-6x-2.$$
Assume first that $t=0$. In this case we need to find the common solutions of
$P(0,x)=0$ and $\partial_x P(0,x)= 0$: the discriminant of $P(0,x)$ is $3^6$ so 
these two polynomials have a common root only if $p=3$ and we already know what 
happens in that case. 

We can thus assume that $t\neq 0$. Consider now the polynomial $(2x^2-6x+4)$:
its roots are $1$ and $2$. A computation shows that $\partial_t P(t,1)$ and 
$\partial_t P(t,2)$ cannot be zero in characteristic $0$ or odd.

We can thus multiply $\partial_x P(t,x)$ times $(2x^2-6x+4)^2$ in order to
eliminate $t$, and we get $\partial_x P(t,x)(2x^2-6x+4)^2=3x^2-4x$. 

Assume $x=0$. In this case $t^2=1$ and $P(1,0)=-1$, so this case cannot arise.

If $p=3$, then again we must have $x=0$ and the preceding reasoning applies.

We can thus assume that $p\neq 3$ and $x=4/3$ and $t^2=10/9$ and again 
$P(4/3,10/9)$ can never be $0$ for any $p$. 

We now will show that the variety consists of precisely two irreducible
components in every characteristic ($\neq 2$). Obviously the abelian component,
being a line, can never split into more than one component. It is thus enough
to prove that $P(t,x)$ is always irreducible. Assume by contradiction that
curve corresponding to $P(t,x)=0$ is not irreducible. Its irreducible 
components must intersect in the projective closure in some singular points, 
since we are considering algebraically closed fields. Similarly, every 
irreducible component must meet the line at infinity. We now observe that our 
curve intersects the line at infinity in two points, one with (homogeneous) 
coordinates equal to $t=1$ and $x=0$ and a second one with coordinates $t=0$ 
and $x=1$. The first point has multiplicity four and is smooth with tangent the 
line at infinity in any characteristic $\neq 2$. The second one has multiplicity 
two and is a singular point, again with tangent the line at infinity. Since the 
first point is smooth, it must belong to a single irreducible component. If 
$p\neq 3$ this irreducible component must also contain the second point at 
infinity, which is the only singular point of the curve. Because of the 
multiplicities of intersection, this irreducible component must be a curve of 
degree at least five. It follows that if there were a second component it 
should be a line, passing through the second point and tangent to the line at 
infinity. This is however absurd since the component cannot be the line at 
infinity. If $p=3$, the component containing the smooth point at infinity 
might have degree $4$, since there are two singular points. We consider then 
its intersection with the line of equation $x=2$. Because of its degree, this 
component meets the line at the singular point $(t=0,x=2)$. If the multiplicity 
of intersection were $4$, then the component could not intersect the line 
elsewhere and thus it could not intersect the remaining irreducible 
component(s) since it could not pass through the second singular 
point which is also on this line. This shows that the component must pass
through the second point at infinity even in this case and the conclusion
follows as in the previous one. 

\begin{center}
  \textsc{Example: $4_1=4/5$ }
 \end{center}

We already discussed the character variety of the figure-eight knot in the
introduction. It is well-known that it consists of just two irreducible
components: the abelian one and Thurston's excellent one. This is also a
consequence of Proposition~\ref{p:irred} since both $\alpha=5$ and 
$(\alpha-1)/2=2$ are prime numbers.

The same type of arguments and computations used in the previous example, show
that in this case, too, there is a single odd prime for which the character
variety ramifies, that is $p=5$. Recall that in this case we have 
$$P(t,x)=(x^2+(1-t^2)x+t^2-1)$$
while the partial derivatives are
$$\partial_t P(t,x)=2t[x-1]$$
and
$$\partial_x P(t,x)= 2x-t^2+1.$$
We see that singular points can only occur for $t=0$ or $x=1$. The latter case, 
is impossible as $P(t,1)=1$ for any $t$. In the former case we must have 
$2x+1=0$ or equivalently $x=-1/2$ since we are assuming $p\neq 2$. This gives 
$P(t,x)=-5/4$, so we have singular points only for $p=5$. This is precisely the 
situation described by Theorem~\ref{th:bridge}. 

To see that the reduction mod $p$ of the polynomial $P(t,x)$ is irreducible for
all $p\ge 3$, one observes that the points at infinity of $P(t,x)$ are both
smooth. As a consequence, $P(t,x)$ can only split if $p=5$ and at most into two
components. Since the degrees of the two irreducible components must be $2$ and 
$1$, the component of degree one is necessarily one of the tangents  at the
singular point $(t=0,x=1)$, however the two tangents do not meet the line at
infinity at the same points as $P(t,x)$. The computational details are left to 
the reader. 

\begin{center}
  \textsc{Example: $7_4=11/15$ }
 \end{center}

The character variety of this twist knot was computed and studied in
\cite{Chu} by Chu who was interested in understanding intersections between
different, non abelian, irreducible components. Indeed, in this case the
variety has two irreducible components, besides the abelian one, as shown by
\cite{Oh,MPL}.

Rewriting the equations in \cite{Chu} using our notation, we have
$$P(t,x)=(-1+2x^2+x^3-x^2t^2)(1+4x-4x^2-x^3+x^4-2xt^2+3x^2t^2-x^3t^2).$$

Theorem~\ref{th:bridge} tells us that $p=3$ and $p=5$ are ramified primes for
this variety. We wish to show that these are the only primes that ramify.

We start by studying each irreducible component separately. We begin with 
$$P_1(t,x)=-1+2x^2+x^3-x^2t^2$$
whose partial derivatives are
$$\partial_t P_1(t,x)=-2tx^2$$
and
$$\partial_x P_1(t,x)= -2xt^2+4x+3x^2.$$
We see that both derivatives are $0$ if $x=0$ but $P_1(t,0)=-1$, so we can
exclude the value $x=0$. It follows that we must have $t=0$ and $4+3x=0$. If
$p=3$ there is no solution, so we can assume $p\neq 3$ and $x=-4/3$. We get
$P_1(0,-4/3)=-5/27$ so, as expected, $p=5$ is a ramified prime. Writing
$P_1(t,x)$ as $-t^2(x-2)^2-4t^2(x-2)-4t^2+(x-2)^3+8(x-2)^2+20(x-2)+50$, one
sees that, in characteristic $p=5$, the tangent at the singular point $t=0$,
$x=-4/3\equiv 2$ has equation $t^2-2(x-2)^2=0$, where coefficients are thought 
mod $5$.

To understand whether this component splits in some characteristic, we now
analyse its points at infinity. There are two such points, both intersecting
the line at infinity with multiplicity $2$. One of them, with homogeneous 
coordinates $t=0$ and $x=1$, is smooth with tangent the line at infinity, while
the other, with homogeneous coordinates $t=1$ and $x=0$, is singular with
tangent of equation $-x^2=0$. Consider now the irreducible component containing
the smooth point at infinity. Since the curve only meets the line of equation
$x=0$ at infinity, such irreducible component passes through both points at
infinity and had degree at least $3$. If it had degree $3$, the second
irreducible component would have degree one and would have to coincide with the
tangent at the singular point, that is the line of equation $x=0$. This is
clearly impossible, so this curve is irreducible in all characteristics. 

We pass now to the second component
$$P_2(t,x)=1+4x-4x^2-x^3+x^4-2xt^2+3x^2t^2-x^3t^2$$
whose partial derivatives are
$$\partial_t P_2(t,x)=2t(-2x+3x^2-x^3)$$
and
$$\partial_x P_2(t,x)= 4-8x-3x^2+4x^3+t^2(-2+6x-3x^2).$$

From the first derivative, we see that either $x\in\{0,1,2\}$ or $t=0$. For
each value of $x$ in $\{0,1,2\}$ we get $P_2(t,x)=1$, so we must have $t=0$.
Computing the resultant of the polynomials $1+4x-4x^2-x^3+x^4$ and 
$4-8x-3x^2+4x^3$ one obtains $3^25^3$: as expected the curve has singular
points in characteristics $p=3$ and $p=5$.

As for the previous component, we show that this component as well is always
irreducible. As before we consider the points at infinity of the curve: again
we have two points. The first one has homogeneous coordinates $t=0$ and $x=1$,
it is smooth with tangent the line at infinity, and the multiplicity of
intersection between the curve and the line at infinity at this point is $2$.
The second point has homogeneous coordinates $t=1$ and $x=0$, it is a singular
point with tangent of equation $x^3=0$ and multiplicity of intersection $3$.
The very same argument seen in the previous case, shows that the irreducible
component containing the smooth point at infinity must pass through the second
one as well and have degree at least $3$. If the curve splits in two or more
irreducible components, the component of degree at most two must be pass
through the singular point at infinity with tangent of equation $x=0$. As a
consequence such component can only be the line $x=0$, with multiplicity one or
two. This is however impossible.

We turn now our attention to the points of intersection between the two
components. Since $x$ cannot be equal to $0$, we can use the $P_1(t,x)$ to
eliminate $t^2$ from $P_2(t,x)$. We obtain the polynomial $x^2-2x+2$ that
always has two roots in any characteristic different from $2$ (cfr.
\cite{Chu}). 

 \begin{center}
  \textsc{Example: $8_9=7/25$}
 \end{center}

From the presentation 
$$
\pi_1(\mathbf{S}^3\setminus 8_9)=\langle  \alpha, \beta\mid  
\beta \alpha^{-1} \beta \alpha \beta^{-1} \alpha^{-1} \beta \alpha \beta^{-1}= 
\alpha^{-1} \beta \alpha \beta^{-1} \alpha^{-1} \beta \alpha \beta^{-1} \alpha  
\rangle
$$
by taking traces and using the coordinates 
$t=\operatorname{tr}_\alpha=\operatorname{tr}_\beta$, 
$x=\operatorname{tr}_{\alpha\beta^{-1}}$
with the help of symbolic software
we get
\medskip

\noindent
$P(t,x)=
{t}^{18}{x}^{3}-6\,{t}^{18}{x}^{2}-9\,{t}^{16}{x}^{4}+12\,{t}^{18}x+47
\,{t}^{16}{x}^{3}+36\,{t}^{14}{x}^{5}-8\,{t}^{18}-66\,{t}^{16}{x}^{2}-
160\,{t}^{14}{x}^{4}-84\,{t}^{12}{x}^{6}-12\,{t}^{16}x+115\,{t}^{14}{x
}^{3}+308\,{t}^{12}{x}^{5}+126\,{t}^{10}{x}^{7}+56\,{t}^{16}+273\,{t}^
{14}{x}^{2}+35\,{t}^{12}{x}^{4}-364\,{t}^{10}{x}^{6}-126\,{t}^{8}{x}^{
8}-232\,{t}^{14}x-928\,{t}^{12}{x}^{3}-441\,{t}^{10}{x}^{5}+266\,{t}^{
8}{x}^{7}+84\,{t}^{6}{x}^{9}-140\,{t}^{14}+187\,{t}^{12}{x}^{2}+1515\,
{t}^{10}{x}^{4}+763\,{t}^{8}{x}^{6}-112\,{t}^{6}{x}^{8}-36\,{t}^{4}{x}
^{10}+756\,{t}^{12}x+630\,{t}^{10}{x}^{3}-1362\,{t}^{8}{x}^{5}-679\,{t
}^{6}{x}^{7}+20\,{t}^{4}{x}^{9}+9\,{t}^{2}{x}^{11}+124\,{t}^{12}-1648
\,{t}^{10}{x}^{2}-1755\,{t}^{8}{x}^{4}+647\,{t}^{6}{x}^{6}+345\,{t}^{4
}{x}^{8}+2\,{t}^{2}{x}^{10}-{x}^{12}-725\,{t}^{10}x+1796\,{t}^{8}{x}^{
3}+1964\,{t}^{6}{x}^{5}-108\,{t}^{4}{x}^{7}-95\,{t}^{2}{x}^{9}-{x}^{11
}+42\,{t}^{10}+1744\,{t}^{8}{x}^{2}-924\,{t}^{6}{x}^{4}-1163\,{t}^{4}{
x}^{6}-27\,{t}^{2}{x}^{8}+11\,{x}^{10}-93\,{t}^{8}x-2206\,{t}^{6}{x}^{
3}+76\,{t}^{4}{x}^{5}+358\,{t}^{2}{x}^{7}+10\,{x}^{9}-118\,{t}^{8}-29
\,{t}^{6}{x}^{2}+1544\,{t}^{4}{x}^{4}+120\,{t}^{2}{x}^{6}-45\,{x}^{8}+
422\,{t}^{6}x+225\,{t}^{4}{x}^{3}-565\,{t}^{2}{x}^{5}-36\,{x}^{7}+25\,
{t}^{6}-560\,{t}^{4}{x}^{2}-201\,{t}^{2}{x}^{4}+84\,{x}^{6}-85\,{t}^{4
}x+326\,{t}^{2}{x}^{3}+56\,{x}^{5}+25\,{t}^{4}+95\,{t}^{2}{x}^{2}-70\,
{x}^{4}-46\,x{t}^{2}-35\,{x}^{3}-6\,{t}^{2}+21\,{x}^{2}+6\,x-1
$

\

It intersects $t=0$ at

\begin{equation}
 \label{eqn:phi_25}
 \left( {x}^{2}+x-1 \right)  \left( {x}^{10}-10\,{x}^{8}+35\,{x}^{6}+
{x}^{5}-50\,{x}^{4}-5\,{x}^{3}+25\,{x}^{2}+5\,x-1 \right) 
\end{equation}
Which corresponds to roots of order 5 and 25, respectively.
This polynomial is congruent to 
$$
(x-2)^{12}\mod 5
$$

Thus the whole $Y(K)$ collapses to a single point when reducing mod $5$.

Unexpectedly, $P(t,x)$ factors non trivially mod $7$ (but not mod $5$):

\begin{align*}
 P(t,x)= & ( {t}^{12}{x}^{2}+3\,{t}^{12}x+{t}^{10}{x}^{3}+4\,{t}^{12}+6\,{t
}^{10}{x}^{2}+{t}^{8}{x}^{4}+6\,{t}^{10}x
+2\,{t}^{8}{x}^{3}+{t}^{6}{x}
^{5}+5\,{t}^{10}
\\
& \qquad
+6\,{t}^{8}{x}^{2}+5\,{t}^{6}{x}^{4}+{t}^{4}{x}^{6}+6
\,{t}^{8}x+4\,{t}^{6}{x}^{3}+{t}^{4}{x}^{5}+{t}^{2}{x}^{7}+2\,{t}^{8}+
{t}^{6}{x}^{2}+4\,{t}^{2}{x}^{6}
\\
& \qquad
+{x}^{8}+4\,{t}^{6}x+2\,{t}^{4}{x}^{3}
+{t}^{2}{x}^{5}+5\,{t}^{6}+4\,{t}^{4}{x}^{2}+{t}^{4}x+4\,{t}^{4}+4\,{t
}^{2}{x}^{2}+4\,{x}^{4}+5\,x{t}^{2}
\\
& \qquad
+4\,{x}^{3}+5\,{t}^{2}+5\,{x}^{2}+2
\,x+6 ) \\
& \times
( {t}^{6}x+5\,{t}^{6}+4\,{t}^{4}{x}^{2}+3\,{t}^{4}
x+3\,{t}^{2}{x}^{3}+6\,{t}^{4}+6\,{x}^{4}+4\,x{t}^{2}+6\,{x}^{3}
\\
& \qquad
+4\,{t
}^{2}+4\,{x}^{2}+3\,x+1 ) \qquad\qquad \mod 7
\end{align*}

Using symbolic software again, we may check that in characteristic zero 
$P(t,x)=0$ has no singular points. In fact, besides characteristic $p=5$ and 
$p=7$, in characteristic $p=23$ one finds that $(t,x)=(2,\pm 6)$ is a singular 
point of the variety. Notice that the discriminant of the Alexander polynomial
is $13225=5^2\, 23^2$, which, by Remark~\ref{Remark:Discriminant}, implies that 
there are singular points with $x=2$ in characteristic $p$ at most for $p=5$ 
and $p=23$.

\medskip

\begin{center}
  \textsc{Example: $11/23$ }
 \end{center}

In this case the order of the homology of the double branched cover is a prime
number and moreover we have that $11=(23-1)/2$ is also prime. We mainly study
this twist knot because its character variety appears as a subvariety of the
Whitehead double we study in Section~\ref{sec:sat}. Indeed, there is a degree-one map
from the Whitehead double to this knot obtained by sending the trefoil
companion onto the trivial knot. Of course the map induces a surjection of
fundamental groups and so an injection of character varieties. 
 
We give here a presentation obtained from a presentation of the Whitehead
link $5_1^2$ by surgery on one of its two components. This will become handy 
when studying the Whitehead double.

We have for the Whitehead link
$$
\pi_1(\mathbf{S}^3\setminus 5_1^2)=
\langle \alpha, b \mid \alpha b^{-1}\alpha^{-1}b\alpha^{-1}b^{-1}\alpha=
b^{-1}\alpha b^{-1}\alpha^{-1}b \alpha^{-1}b^{-1}\alpha b\rangle ,
$$
and for the twist knot
$$
\pi_1(\mathbf{S}^3\setminus 11/23)=
\langle\alpha,\beta \mid u=\beta^{-1}\alpha, b=u^6, 
u=\alpha b^{-1}\alpha^{-1}b \alpha^{-1}b^{-1}\alpha b, u=b^{-1}ub \rangle ,
$$
where the group of the twist knot is obtained by killing by Dehn surgery the
slope $u^6b^{-1}$, where $u=\alpha b^{-1}\alpha^{-1}b \alpha^{-1}b^{-1}\alpha
b$, so that the relation of the Whitehead link group becomes 
$u=b^{-1}ub$ which is redundant in the quotient group; finally one can recover
a standard $2$-bridge presentation for the twist knot by observing that
$\beta=\alpha u^{-1}$ is conjugate to $\alpha$ hence a meridian that can be 
chosen as a generator of the group.

Hence we take coordinates 
$$
t=\operatorname{tr}_{\alpha}=\operatorname{tr}_{\beta}, \qquad 
x=\operatorname{tr}_{\alpha^{-1}\beta},
$$
and by taking traces on the equality
$$
\alpha^{-1} ( \beta\alpha^{-1} )^5 \beta\alpha ( \beta^{-1}\alpha)^5 \beta
=
(\beta\alpha^{-1})^5(\beta\alpha)(\beta^{-1}\alpha) ^5
$$
we get
$$
(x-2) P(t,x)=0
$$
where
\begin{multline*}
P(t,x)=
{t}^{2}{x}^{10}-{t}^{2}{x}^{9}-{x}^{11}-8\,{t}^{2}{x}^{8}-{x}^{10}+7\,
{t}^{2}{x}^{7}+10\,{x}^{9}+22\,{t}^{2}{x}^{6}+9\,{x}^{8}
\\
-16\,{t}^{2}{x
}^{5}
-36\,{x}^{7}-24\,{t}^{2}{x}^{4}-28\,{x}^{6}+13\,{t}^{2}{x}^{3}+56
\,{x}^{5}
\\
+9\,{t}^{2}{x}^{2}+35\,{x}^{4}-3\,{t}^{2}x-35\,{x}^{3}-15\,{x
}^{2}+6\,x+1 .
\end{multline*}
with
$$
P(0,x)\equiv -(x-2)^{11}\mod 23
$$
An analysis similar to those seen for the first three examples shows that
$p=23$ is the only prime that ramifies for this knot.


\section{A pretzel example}
\label{sec:pretzel}

For a Montesinos knot $K$ that is not a $2$-bridge one, 
$Y^{nab}(K)\neq\emptyset$, as it contains irreducible representations of the 
double branched cover $M_2$. Since $M_2$ is Seifert fibered, those 
representations from $\pi_1(M_2)$ to $\mathrm{PSL}_2(\mathbb{C})$ map the 
fibre to the identity (the centre of $\mathrm{PSL}_2(\mathbb{C})$ is trivial), 
hence they are irreducible representations of the $2$-orbifold, the space of 
fibres. For a Montesinos knot with $k$ tangles, this will yield components in 
$Y^{nab}(K)$ of dimension up to $k-3$, see for instance \cite{PP}. For a 
pretzel knot, we have $k=3$.

 \begin{center}
  \textsc{Example: $8_5=P(3,3,2)$ }
 \end{center}

Here $\Delta_K(-1)=21$, hence $\vert Y^{ab}(K)\vert= 10$. If we restrict these 
characters to $M_2$, they are characters of ten cyclic representations 
$H_1(M_2,\ZZ)\to\operatorname{PSL}_2(\CC)$, one of them has order $3$, three  
have order $7$, and six have order $21$ (the cardinality is one half of 
Euler's $\phi$).

Since $K$ is the pretzel $(2,3,3)$, elements  in $Y^{nab}(K)$ are 
characters of representations in $\mathrm{SU}(2)$.

We will consider reductions mod $p$ for $p=3,7$.

\bigskip

 \begin{figure}[h]
 \begin{center}
  \begin{tikzpicture}[scale=.6,line cap=round]
\begin{scope}
    \braid[line width=1pt,  color=black] s_1^{-1} s_1^{-1} ;
\end{scope}
\begin{scope}[shift={(2,0)}]
    \braid[line width=1pt,  color=black] s_1^{-1} s_1^{-1} s_1^{-1};
\end{scope}
\begin{scope}[shift={(4,0)}]
    \braid[line width=1pt,  color=black] s_1^{-1} s_1^{-1} s_1^{-1};
\end{scope}
   \draw [thick, color=black] (3,0) arc [radius=.5, start angle=0, end angle= 180];
   \draw [thick, color=black] (5,0) arc [radius=.5, start angle=0, end angle= 180];
   \draw [thick, color=black] (6,0) arc [radius=1, start angle=0, end angle= 90];
    \draw [thick, color=black] (2,1) arc [radius=1, start angle=90, end angle= 180];
    \draw [thick, color=black] (2,1) -- (5,1);
    \draw [thick, color=black] (1,-2.5) -- (1,-3.5);
    \draw [thick, color=black] (2,-2.5) -- (2,-3.5);    
   \draw [thick, color=black] (4,-3.5) arc [radius=.5, start angle=180, end angle= 360];
   \draw [thick, color=black] (2,-3.5) arc [radius=.5, start angle=180, end angle= 360];
   \draw [thick, color=black] (1,-3.5) arc [radius=1, start angle=180, end angle= 270];
     \draw [thick, color=black] (2,-4.5) -- (5,-4.5);
     \draw [thick, color=black] (5,-4.5) arc [radius=1, start angle=270, end angle= 360];
\draw [color=blue] (3.5,-4.9)--(3.5,-4.6);  
\draw [->, color=blue] (3.5,-4.4)--(3.5,-4.1);  
\draw [color=blue] (4.5,-4.4)--(4.5,-4.1); 
\draw [->, color=blue] (4.5,-3.9)--(4.5,-3.6); 
\draw [color=blue] (2.5,-3.6)--(2.5,-3.9); 
\draw [->, color=blue] (2.5,-4.1)--(2.5,-4.4); 
 \node [color=blue] at (3.85,-4.85) {$\alpha$};
 \node [color=blue] at (2.5,-3.3) {$\beta$};
 \node [color=blue] at (4.5,-3.3) {$\gamma$};
\draw [thick, color=red] (5,-1.25)--(6,-1.25); 
 \end{tikzpicture}
 \end{center}
  \caption{The knot $8_{5}$. The generators $\alpha$, $\beta$ and $\gamma$ of 
its fundamental group and a tunnel.} 
 \label{fig:8_5}
\end{figure}
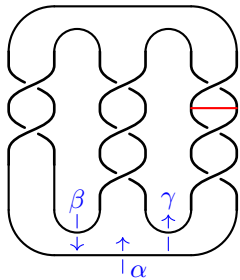

The group of $8_5$ is generated by three meridians, $\alpha$, $\beta$ and 
$\gamma$ in Figure~\ref{fig:8_5}. Since it has tunnel number one, the group can 
also be generated by two elements. Those are $\gamma$ and 
$a=\gamma^{-1}\alpha^{-1} (\beta\gamma)^2 $, see Appendix~\ref{appendix:8_5}. 
Hence we choose coordinates:
\begin{eqnarray*}
t& = \operatorname{tr}_\gamma\\
x&= \operatorname{tr}_a\\
y& = \operatorname{tr}_{\gamma a}
\end{eqnarray*}

With the help of symbolic software, there are 5 components:

\begin{enumerate}
 \item Abelian component
\begin{eqnarray*}
x = & t^2-2 \\
y= & t^3-3 t,
\end{eqnarray*}
because $x$ and $y$ correspond to traces of elements whose abelianisations are 
2 and 3 respectively.

\item A ``trefoil'' component:
$$
y=x-1=0
$$
namely, a component that looks like the variety of irreducible characters of 
the trefoil knot and which is  induced by an epimorphism onto the trefoil 
group.

\item The canonical component. 

\begin{eqnarray*}
t{x}^{3}-{t}^{2}y-t{y}^{2}-{x}^{2}y+2\,tx+2\,yx-t-y
 & =0  ,
\\
{t}^{2}yx+{y}^{2}tx-2\,t{x}^{2}-t{y}^{2}-{x}^{2}y-{y}^{3}-tx+yx+t+2\,y
 & =0  ,
\\
t{y}^{3}x+t{x}^{2}y-{y}^{2}{x}^{2}-{y}^{4}-2\,txy-2\,{x}^{3}-x{y}^{2}-
ty+2\,{x}^{2}+3\,{y}^{2}+4\,x-2
& = 0.
\end{eqnarray*}

\item An exotic component:
$$
x+1=ty+{y}^{2}-2=0
$$

\item A second exotic component:
\begin{eqnarray*}
{t}^{2}{x}^{2}y-t{x}^{3}-2\,{y}^{2}tx+{x}^{2}y+{y}^{3}+2\,tx-t-3\,y  & =0 , \\
{t}^{3}yx+{t}^{2}x{y}^{2}+t{y}^{3}x-{t}^{3}y-{t}^{2}{x}^{2}-2\,{t}^{2}
{y}^{2}
-t{x}^{2}y-t{y}^{3}-{y}^{2}{x}^{2}
 \qquad \qquad \quad
 &\\
-{y}^{4}+{t}^{2}+4\,ty+4\,{y}
^{2}  & =0 ,  \\
 t{x}^{5}y+t{x}^{4}y-{x}^{6}-{x}^{4}{y}^{2}+2\,{t}^{2}x{y}^{2}-4\,ty{x}
 ^{3}+2\,t{y}^{3}x-{x}^{5}-{x}^{3}{y}^{2}  \qquad \qquad \quad
  &\\
 -{t}^{3}y-{t}^{2}{y}^{2} 
 -7\,t{
 x}^{2}y-2\,t{y}^{3}+6\,{x}^{4}+{y}^{2}{x}^{2}
  -2\,{y}^{4
  }  
  +{t}^{2}x+6\,{
 x}^{3}
   \qquad  \quad
  &\\
 +3\,x{y}^{2}+{t}^{2}+8\,ty-8\,{x}^{2}+7\,{y}^{2}-8\,x-1 & = 0
\end{eqnarray*}

\end{enumerate}

Next we describe the intersection of each component  with $t=0$:
\begin{enumerate}
\item Abelian: $ x=-2$, $y=0$. This is the point in $Y^{triv}(K)$

\item Trefoil: $y=0$, $x=1$. This point lies in $Y^{ab}(K)$ (it corresponds to 
the element of order 3).

\item Canonical: it is the union of two sets:
\begin{itemize}
 \item  $y=x^3-x^2-2 x+1=0$ that lies in $Y^{ab}(K)$ (3 points of order 7), and
 \item $x-1=y^2-2=0$, two points in $Y^{nab}(K)$.
\end{itemize}
\item Exotic:  $x+1=y^2-2=0$, two points in $Y^{nab}(K)$.

\item Second Exotic:  
$y=-{x}^{6}-{x}^{5}+6\,{x}^{4}+6\,{x}^{3}-8\,{x}^{2}-8\,x-1=0$. 
Six points in $Y^{ab}(K)$ (of order 21).
\end{enumerate}
Notice that the $4$ points in $Y^{nab}(K)$ correspond to two non conjugate 
representations in $\mathrm{PSL}_2(\CC)$ ($x$ is already a variable 
of the characters in $\mathrm{PSL}_2(\CC)$ but $y$ is not).

This knot is $\pi$-spherical, so the lift of its holonomy in 
$\mathrm{SU}(2)\times \mathrm{SU}(2)$ projects to the tetrahedral group in one 
$\mathrm{SU}(2)$ factor (for the basis $2$-orbifold), and to the binary 
dihedral group of order $28$ in the other $\mathrm{SU}(2)$ factor. This 
dihedral group in one of the factors corresponds to an action of $\S^1$ 
that is preserved up to orientation, yielding the orbifold Seifert fibration.

Regarding the ramifications, we obviously have the collapses of $Y^{ab}(K)$ 
mod 3 and  mod 7:
\begin{itemize}
\item When reducing mod 3, the point of order 3 becomes trivial 
(in $Y^{triv}(K)$). The 6 points of order 21 are identified to some of the 3 
points of order 7.
\item When reducing mod 7, the points of order 7 become also trivial (in 
$Y^{triv}(K)$). The 6 points of order 21 become the point of order 3.
\end{itemize}

\section{A $\pi$-hyperbolic knot}
\label{sec:pihyperbolic}

The knot $8_{18}$ is $\pi$-hyperbolic, that is its double branched cover is 
hyperbolic. Here the canonical component splits into two irreducible components
in characteristic $3$.

 \begin{figure}[h]
 \begin{center}
  \begin{tikzpicture}[scale=.6,line cap=round]
\begin{scope}
  \begin{scope}
 \draw [->] (0.1,1.1) -- (0.1,1.4); 
 \draw (0.1, 0.7)--(0.1,0.9); 
  \draw [->] (0.1,2.1) -- (0.1,2.4); 
  \draw (0.1, 1.7)--(0.1,1.9);
  \draw [->] (0.1,3.1) -- (0.1,3.4); 
  \draw (0.1, 2.7)--(0.1,2.9);-.77);
  \node at (-0.35,1.3) {$\gamma$};
  \node at (-0.35,2.3) {$\beta$};
  \node at (-0.35,3.3) {$\alpha$};
   \braid[line width=1pt,  line cap=round , color=black, rotate=90] s_2^{-1} ;
 \end{scope}
   \begin{scope}[shift={(5.04,1.91)}]
   \braid[line width=1pt,  color=black, rotate=225] s_2;
 \end{scope}
   \draw [thick, color=black] (1.5,1) arc [radius=.5, start angle=90, end angle= 45];
 \draw [thick, color=black] (1.5,2) arc [radius=1.5, start angle=90, end angle= 45]; 
 \draw [thick, color=black] (1.5,3) arc [radius=2.5, start angle=90, end angle= 45]; 
 \draw [thick, color=black] (2.91,-0.2) arc [radius=.5, start angle=45, end angle= 0];
   \draw [thick, color=black] (3.61,.51) arc [radius=1.5, start angle=45, end angle= 0]; 
   \draw [thick, color=black] (4.33,1.21) arc [radius=2.5, start angle=45, end angle= 0];
\end{scope} 
\begin{scope}[shift={(2.05,-.55)}, rotate=270]
  \begin{scope}
   \braid[line width=1pt,  color=black, rotate=90] s_2^{-1} ;
 \end{scope}
   \begin{scope}[shift={(5.04,1.91)}]
   \braid[line width=1pt,  color=black, rotate=225] s_2;
 \end{scope}
   \draw [thick, color=black] (1.5,1) arc [radius=.5, start angle=90, end angle= 45];
 \draw [thick, color=black] (1.5,2) arc [radius=1.5, start angle=90, end angle= 45]; 
 \draw [thick, color=black] (1.5,3) arc [radius=2.5, start angle=90, end angle= 45]; 
 \draw [thick, color=black] (2.91,-0.2) arc [radius=.5, start angle=45, end angle= 0];
   \draw [thick, color=black] (3.61,.51) arc [radius=1.5, start angle=45, end angle= 0]; 
   \draw [thick, color=black] (4.33,1.21) arc [radius=2.5, start angle=45, end angle= 0];
\end{scope} 
\begin{scope}[shift={(1.5,-2.61)}, rotate=180]
  \begin{scope}
   \braid[line width=1pt,  color=black, rotate=90] s_2^{-1} ;
 \end{scope}
   \begin{scope}[shift={(5.04,1.91)}]
   \braid[line width=1pt,  color=black, rotate=225] s_2;
 \end{scope}
   \draw [thick, color=black] (1.5,1) arc [radius=.5, start angle=90, end angle= 45];
 \draw [thick, color=black] (1.5,2) arc [radius=1.5, start angle=90, end angle= 45]; 
 \draw [thick, color=black] (1.5,3) arc [radius=2.5, start angle=90, end angle= 45]; 
 \draw [thick, color=black] (2.91,-0.2) arc [radius=.5, start angle=45, end angle= 0];
   \draw [thick, color=black] (3.61,.51) arc [radius=1.5, start angle=45, end angle= 0]; 
   \draw [thick, color=black] (4.33,1.21) arc [radius=2.5, start angle=45, end angle= 0];
\end{scope} 
\begin{scope}[shift={(-.56,-2.06)}, rotate=90]
  \begin{scope}
   \braid[line width=1pt,  color=black, rotate=90] s_2^{-1} ;
 \end{scope}
   \begin{scope}[shift={(5.04,1.91)}]
   \braid[line width=1pt,  color=black, rotate=225] s_2;
 \end{scope}
   \draw [thick, color=black] (1.5,1) arc [radius=.5, start angle=90, end angle= 45];
 \draw [thick, color=black] (1.5,2) arc [radius=1.5, start angle=90, end angle= 45]; 
 \draw [thick, color=black] (1.5,3) arc [radius=2.5, start angle=90, end angle= 45]; 
 \draw [thick, color=black] (2.91,-0.2) arc [radius=.5, start angle=45, end angle= 0];
   \draw [thick, color=black] (3.61,.51) arc [radius=1.5, start angle=45, end angle= 0]; 
   \draw [thick, color=black] (4.33,1.21) arc [radius=2.5, start angle=45, end angle= 0];
\end{scope}
\draw [color=blue] (.75,4)--(.75,-6);
\node at (2, -6.35) {$\sigma_1 $};
 \draw [thin, color=blue, ->] (.25,-6) arc [radius=.5, start angle=190, end angle= 350];  
 \draw [thin, color=black, ->] (1.25,-1.5) arc [radius=.5, start angle=0, end angle= 90];  
\node at (1.50, -1.0) {$\phi $}; 
\fill[red]  (-1,1.6) circle (.1); 
\fill[red]  (2,-4.5) circle (.1);
%
%
%

; 
 
 \end{tikzpicture}
 \end{center}
  \caption{The knot $8_{18}$. 
 The period $\phi$, the axis of the 
strong inversion $\sigma_1$ and the fixed points of the
orientation reversing involution $\sigma_2$ marked as thick red dots} 
\label{fig:8_18}
\end{figure}
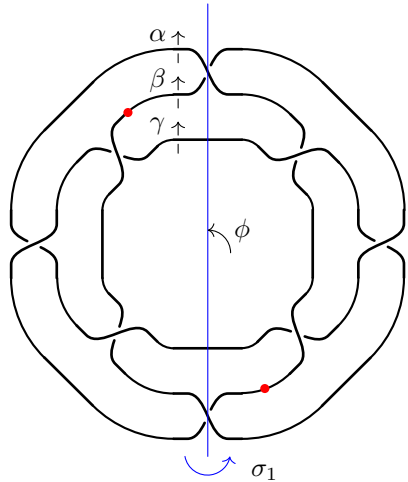

The knot $8_{18}$ has a period $\phi$, corresponding to the rotation of order 
four in Figure~\ref{fig:8_18}. The symmetry group of the knot is generated by
the rotation together with two involutions $\sigma_1$ and $\sigma_2$. The 
{
orientation preserving involution $\sigma_1$ is a rotation of angle $\pi$ about 
an axis represented as a blue line in Figure~\ref{fig:8_18}: it is a strong
inversion of the knot. The orientation-reversing involution $\sigma_2$ is 
reflection with two fixed points belonging to the knot (after a suitable 
conformal change of coordinates of $\S^3=\mathbb{R}^3$ sending one of the two 
fixed points to $\infty$ and the other to the origin, $\sigma_2$ can be seen as 
the linear map with matrix minus the identity).
}

Here $\vert\Delta_K(-1)\vert=45$, and $H_1(M_2)=\ZZ/5\oplus\ZZ/3\oplus\ZZ/3$.

We sketch the computation of $X(8_{18})$ in Appendix~\ref{appendix:8_18}, here 
we just describe it. We use the following coordinates for $X(8_{18})$:
$$
\begin{array}{l}
 t  = tr_{\alpha }= tr_{\beta}= tr_{\gamma}\\
 x  = tr_{\alpha\beta^{-1}} \\
 y = tr_{\beta^{-1}\gamma} \\
 z  = tr_{\alpha\gamma} \\
 w  = tr_{\alpha\beta^{-1}\gamma}
 \end{array}
$$
We next list the 11 components of $X(8_{18})$.

\paragraph{Abelian component}

\begin{equation}
 \label{eqn:abelian}
 \langle {-t+w,x-2,y-2,-{t}^{2}+z+2} \rangle 
\end{equation}
The symmetries act trivially here. When $t=0$ we get the character that is 
trivial on $M_2$, the double branched covering.

\paragraph{The canonical component}
 \begin{equation}
\label{eqn:canonical}
\langle {tw+{y}^{2}+ \left( -{t}^{2}-1 \right) y+1,-{t}^{2}+y+z,5\,{t}
^{2}-4\,tw+{w}^{2}-6,-{t}^{2}+x+y-1}\rangle 
 \end{equation}
 It can also be written as 
 $$
 \begin{array}{c}
  tw+{y}^{2}+ \left( -{t}^{2}-1 \right) y+1=0\\
  z= t^2-y\\
  x=z+1\\
  5\,{t}
^{2}-4\,tw+{w}^{2}=6
 \end{array}
 $$
It is a double branched covering of the curve described in \cite{HLMlinks} obtained 
by considering only the variables $t$ and $y$ (we can get rid of $x$ and $z$,
the projection is the elimination of $w$). 

This component is fixed point-wise by the period $\phi$ and by $\sigma_1$. The 
involution $\sigma_2$ preserves the component but permutes the holonomy of the 
complete structure with its complex conjugate. 

The intersection with $t=0$ yields precisely the 4 lifts of the hyperbolic 
holonomy of the orbifold ($t=0$, $y=(1\pm\sqrt{-3})/2$, $w=\pm\sqrt{6}$). The 4 
points correspond to complex conjugation ($y=(1\pm\sqrt{-3})/2$) and a choice 
of sign ($w=\pm\sqrt{6}$) as the abelianisation of the orbifold group is 
$\mathbb Z/ 2$ (it only concerns the variable $w$ because $t=0$). When reducing 
mod 3, these 4 points become a single point.

It should be noticed that the last equation
$$
  5\,{t}
^{2}-4\,tw+{w}^{2}=6
$$
factors mod 3  as the product of two lineal equations, hence this component 
splits into 2 components when reducing mod 3.

\paragraph{Four ``trefoil-like'' components}
There are four components that look like the the variety of characters of the 
trefoil:
 \begin{eqnarray}
  \langle {-1+z,-t+w,x-2,-{t}^{2}+y+1}\rangle  \label{eqn:trefoil1}\\
  \langle {-1+z,-t+w,y-2,-{t}^{2}+x+1}\rangle \label{eqn:trefoil2} \\
  \langle {-1+z,-{t}^{3}+2\,t+w,-{t}^{2}+x+1,-{t}^{2}+y+1}\rangle \label{eqn:trefoil3} \\
  \langle {-{t}^{3}+2\,t+w,-{t}^{2}+x+1,-{t}^{2}+y+1,-{t}^{2}+z+2} \rangle  \label{eqn:trefoil4}
\end{eqnarray}
They come from a surjection from $\pi_1(\mathbf{S}^3\setminus 8_{18})$ to the 
group of the trefoil. Namely, if we add the relation $\alpha=\beta$ then we get 
the group of the trefoil, and this corresponds to the 
component~\eqref{eqn:trefoil1}.

For computational purposes, it is useful to give a parameterisation of these 
components:
$$
\begin{array}{lllll}
 x= 2, & y= t^2-1, & z=1, & w= t, &\eqref{eqn:trefoil1} \\
 x=t^2-1, & y=2, & z=1, & w= t, &\eqref{eqn:trefoil2} \\
 x=t^2-1, & y= t^2-1, & z=1, & w= t^3-2t, &\eqref{eqn:trefoil3} \\
 x=t^2-1, & y= t^2-1, & z=t^2-2, & w= t^3-2t, &\eqref{eqn:trefoil4}
\end{array}
$$

The period (the symmetry of order 4) permutes \eqref{eqn:trefoil1} with  
\eqref{eqn:trefoil3}, and \eqref{eqn:trefoil2} with \eqref{eqn:trefoil4}. 
The square of the period fixes point-wise these components.

The involution $\sigma_1$ preserves \eqref{eqn:trefoil1} and  
\eqref{eqn:trefoil3} and permutes \eqref{eqn:trefoil2} with 
\eqref{eqn:trefoil4}.

The involution $\sigma_2$ permutes \eqref{eqn:trefoil1} and 
\eqref{eqn:trefoil2} and preserves the components \eqref{eqn:trefoil3} and 
\eqref{eqn:trefoil4}. 
Thus the group of symmetries acts transitively on these components.

The intersection of the 4 components is the point $x=y=2$, $z=1$, 
$w=t=\pm\sqrt 3$, which also lies in the abelian component (the discriminant of 
the Alexander polynomial vanishes).

When $t=0$, we get the conjugacy classes of 4 representations of order $3$ of 
$M_2$, as $H_1(M_ 2)\cong \mathbb{Z}/3\oplus \mathbb{Z}/3\oplus \mathbb{Z}/5$.

\paragraph{The figure eight knot component}
This component looks like the character variety of the figure eight knot and 
comes from a surjection of fundamental groups. The figure eight knot is in fact 
the quotient of the $8_{18}$ by $\phi^2$, the square of the period. 
The ideal is:
\begin{equation}
\label{eqn:fig8}
 \langle {-t+w,{y}^{2}+ \left( -{t}^{2}+1 \right) y+{t}^{2}-1,-{t}^{2}+
y+z,-{t}^{2}+x+y+1}\rangle
\end{equation}
It can also be presented as
$$
\begin{array}{c}
 w=t\\
 (y-1) t^2= y^2+y-1\\
 z=t^2-y\\
 x\, y = x+ y
\end{array}
$$
It is point-wise invariant by the period  and by the involution $\sigma_1$ and 
preserved by $\sigma_2$, but $\sigma_2$ permutes $x$ and $y$. In particular 
$\sigma_2$ swaps the holonomy of the figure eight knot with its complex 
conjugate.

When $t=0$, $y^2+y-1=0$ corresponds to the two conjugacy classes of nontrivial 
representations of $\mathbb{Z}/5$.

\paragraph{Four more components}

When $t=0$, there are still 16 conjugacy classes of representations of $M_2$ of 
order 15. They belong to 4 components. The first one is:
\begin{multline}
 \label{eqn:exotic1}
\langle {x}^{2}+ \left( -{t}^{2}+1 \right) x+{t}^{2}-1,
\\
\left( {t}^{2
}-1 \right) x+y-z-{t}^{4}+3\,{t}^{2}-1, 
\\
\left( -{t}^{2}+1 \right) xz+
 \left( {t}^{2}+1 \right) x+{z}^{2}+ \left( {t}^{4}-4\,{t}^{2}+1
 \right) z-{t}^{4}+3\,{t}^{2}-1,
 \\
 -{t}^{3}+xt-zt+3\,t+w\rangle 
\end{multline}
The second one is:
\begin{multline}\label{eqn:exotic2}
\langle {y}^{2}+ \left( -{t}^{2}+1 \right) y+{t}^{2}-1, 
\\
w-xt+ \left( -
{t}^{3}+2\,t \right) y+{t}^{5}-4\,{t}^{3}+4\,t,
\\
{x}^{2}+ \left( {t}^{2}
-1 \right) xy+ \left( -{t}^{4}+2\,{t}^{2}-1 \right) x+ \left( -{t}^{4}
+2\,{t}^{2}+1 \right) y+{t}^{6}-4\,{t}^{4}+4\,{t}^{2}-1,
\\
xy+ \left( -{t
}^{2}+1 \right) x+ \left( -{t}^{4}+3\,{t}^{2}-1 \right) y+z+{t}^{6}-5
\,{t}^{4}+6\,{t}^{2}-1 \rangle
%
%
%
\end{multline}
The third one
\begin{multline}\label{eqn:exotic3}
 \langle {z}^{2}+ \left( -{t}^{2}-1 \right) z+2\,{t}^{2}-1,
 \\
 w+xzt-2\,xt
+ \left( -{t}^{5}+4\,{t}^{3}-4\,t \right) z+2\,{t}^{5}-7\,{t}^{3}+5\,t
,
\\
{x}^{2}+ \left( -{t}^{2}+1 \right) xz+ \left( {t}^{2}-1 \right) x+
 \left( {t}^{4}-2\,{t}^{2}-1 \right) z-2\,{t}^{4}+5\,{t}^{2}-1,
 \\
 xz-x+y+
 \left( -{t}^{4}+3\,{t}^{2}-1 \right) z+2\,{t}^{4}-6\,{t}^{2}+1
\rangle 
\end{multline}
The last component is:
\begin{multline}\label{eqn:exotic4}
  \langle
  {t}^{2}-y-z,
  \\
  -z{x}^{2}+x{y}^{2}+xyz+x{z}^{2}+{x}^{2}-4\,xy-xz-2\,{y}^{
2}-yz-{z}^{2}+5\,y+2\,z-1,
\\
{x}^{2}yz-{x}^{2}y-z{x}^{2}-4\,xyz-x{z}^{2}+
4\,xy+6\,xz+3\,yz+{z}^{2}-x-2\,y-5\,z,
  \\
  -{x}^{2}{z}^{2}+x{z}^{3}+2\,z{x}
^{2}+y{z}^{2}-{z}^{3}-{x}^{2}+xy-4\,xz-2\,yz+x-2\,y+4\,z+1,
\\
{z}^{2}yx-4
\,xyz-2\,x{z}^{2}+{y}^{2}z-y{z}^{2}+3\,xy+6\,xz-2\,{y}^{2}+2\,yz+2\,{z
}^{2}-2\,x+2\,y-5\,z-1  ,
\\
-{x}^{2}{z}^{2}+{z}^{2}{y}^{2}+2\,{z}^{3}y+{z}^
{4}+2\,z{x}^{2}-2\,xyz+x{z}^{2}-{y}^{2}z-7\,y{z}^{2}
-5\,{z}^{3}-{x}^{2
}
+3\,xy+xz
\\
\qquad\qquad\qquad\qquad\qquad
-2\,{y}^{2}+5\,yz+6\,{z}^{2}-2\,x+4\,y-z-1,
\\
2\,{x}^{2}{z}^{2}
+z{y}^{3}-3\,{z}^{3}y-2\,{z}^{4}-4\,z{x}^{2}+3\,xyz-3\,x{z}^{2}-2\,{y}
^{3}-8\,{y}^{2}z+6\,y{z}^{2}+9\,{z}^{3}+2\,{x}^{2}-5\,xy
\\
\qquad\qquad\qquad\qquad\qquad
+16\,{y}^{2}+
10\,yz-5\,{z}^{2}+5\,x-22\,y-10\,z+4,
\\
x{z}^{2}t-t{y}^{2}z-2\,{z}^{2}yt-
{z}^{3}t-2\,xzt+2\,t{y}^{2}+10\,zyt+5\,{z}^{2}t-12\,ty-10\,zt+7\,t+w
\rangle
\end{multline}
Each component intersects $t=0$ in 4 points, corresponding to representations 
of $M_2$ or order 15. The action of the symmetry group on the components is:
$$
\begin{array}{c}
 \sigma_1 \\
 \begin{array}{rcl}
  \eqref{eqn:exotic1}& \mapsto & \eqref{eqn:exotic1} \\
  \eqref{eqn:exotic2}& \mapsto & \eqref{eqn:exotic3} \\
  \eqref{eqn:exotic3}& \mapsto & \eqref{eqn:exotic2} \\
  \eqref{eqn:exotic4}& \mapsto & \eqref{eqn:exotic4} \\
 \end{array}
\end{array}
\qquad
\begin{array}{c}
 \sigma_2 \\
 \begin{array}{rcl}
  \eqref{eqn:exotic1}& \mapsto & \eqref{eqn:exotic2} \\
  \eqref{eqn:exotic2}& \mapsto & \eqref{eqn:exotic1} \\
  \eqref{eqn:exotic3}& \mapsto & \eqref{eqn:exotic4} \\
  \eqref{eqn:exotic4}& \mapsto & \eqref{eqn:exotic3} \\
 \end{array}
\end{array}
\qquad
\begin{array}{c}
 \phi \\
 \begin{array}{rcl}
  \eqref{eqn:exotic1}& \mapsto & \eqref{eqn:exotic4} \\
  \eqref{eqn:exotic2}& \mapsto & \eqref{eqn:exotic3} \\
  \eqref{eqn:exotic3}& \mapsto & \eqref{eqn:exotic2} \\
  \eqref{eqn:exotic4}& \mapsto & \eqref{eqn:exotic1} \\
 \end{array}
\end{array}
$$
So the four components are equivalent.

\medskip

\paragraph{Characters of $Y(8_{18})=Y^{triv} (8_{18})\cup Y^{ab} 
(8_{18})\cup Y^{nab} (8_{18})$}
\begin{itemize}
\item $Y^{triv} (8_{18})$ has a single point, with coordinates $t=w=0$, $x=y=2$,
and $z=-2$.

\item $Y^{ab} (8_{18})$ has $\frac{\vert\Delta(-1)\vert -1}{2}=22$ points. When 
looking at the order of the induced representation of 
$\pi_1(M_2)$ in $\operatorname{PSL}_2(\mathbb{C})$, they are distributed as 
follows: the 2 points of order 5 lie in the figure eight knot component, the 4 
points of order 3 lie in trefoil components, one per component, and the 16 
points of order 15 on the exotic components (4 per component). In 
characteristic 3, $Y^{ab} (8_{18})$ has $(5-1)/2=2$ points, in characteristic 
5, $(9-1)/2=4$ points.

\item $Y^{nab} (8_{18})$ has $4$ points, all of them in the canonical component:
two lifts of the orbifold holonomy and their complex conjugate.

\end{itemize}

When reducing mod 3:  The 4 points of $Y^{ab} (8_{18})$ in the trefoil 
components collapse to $Y^{triv} (8_{18})$. The 2 points in the figure eight 
knot component remain as different points. On each exotic component, the 4 
points become (the same) 2 points, that are also the 2 points for the figure 
eight component. The four points in $Y^{nab} (8_{18})$ become a single one, but 
different from the previous ones.

When reducing mod 5: The 4 points of $Y^{ab} (8_{18})$ in the trefoil 
components remain as 4 different points. The four points of $Y^{ab} (8_{18})$ 
in each exotic component become a single one, in fact the same as a point in 
one of the trefoil components, and points in different exotic components go to 
different trefoil components. The 2 points in the figure eight component 
collapse to $Y^{triv} (8_{18})$. The 4 points in $Y^{nab} (8_{18})$ remain as 
different points.


\section{Torus Knots}
\label{sec:torus}

For $m\le n\in \mathbb N$ coprime, the $(m,n)$-torus knot is denoted by
$T(m,n)$. Since
$$\pi_1(\mathbf{S}^3\setminus T(m,n))\cong \langle a,b\mid a^m=b^n\rangle,$$
irreducible representations map the central element $a^m=b^n$ to the centre
$\{\pm\operatorname{Id} \}$ (and only to $-\operatorname{Id}$ for $m=2$).
Thus, besides the abelian one, the components of the variety of characters are
determined by the conjugacy classes of elements in $\SLL_2$ to which $a$ and
$b$ are mapped. Therefore, when $p$ divides either $m$ or $n$, the reduction
mod $p$ ramifies, so that some components get identified. To describe precisely
this ramification, we introduce the following polynomials:

\begin{Definition}
\label{Def:PsiPhi}
For $k\in\mathbb Z$, $k\geq 1$, we define $\Phi_k,\Psi_k\in\ZZ[x]$  to be the
polynomials determined by the condition:
$$
\Phi_k(\lambda+\frac{1}{\lambda})=\begin{cases}
    \dfrac{\lambda^k-1}{\lambda-1} \lambda^{(-k+1)/2} & k\textrm{ odd}  \\[7pt]
    \dfrac{\lambda^k-1}{\lambda^2-1} \lambda^{-k/2+1} & k\textrm{ even} 
                                \end{cases}
$$
and
$$
\Psi_k(\lambda+\frac{1}{\lambda})=\begin{cases}
    \dfrac{\lambda^k+1}{\lambda+1} \lambda^{(-k+1)/2} & k\textrm{ odd}  \\[7pt]
    ({\lambda^k+1}) \lambda^{-k/2} & k\textrm{ even}  
                                \end{cases}
$$
for every $\lambda\in\CC^*$.
\end{Definition}

\begin{Remark}
These polynomials satisfy the following property: for
$A\in\operatorname{SL}_2(\mathbb{K})$ with $\mathbb{K}$ an algebraically closed
field of characteristic 0 or coprime with $k$:
\begin{itemize}
\item $\Phi_k(\operatorname{tr}(A))=0$ if and only if
$A^k=\operatorname{Id}$ and $A\neq \pm \operatorname{Id}$;
\item $\Psi_k(\operatorname{tr}(A))=0$ if and only if
$A^k=-\operatorname{Id}$ and $A\neq - \operatorname{Id}$.
\end{itemize}
\end{Remark}

This remark follows easily from viewing $\lambda$ in
Definition~\ref{Def:PsiPhi} as an eigenvalue of a matrix with determinant $1$.
The following formulae are going to be useful later:
\begin{eqnarray}
\label{eqn:psi2ktraceAk}
\Psi_{2k}(\operatorname{tr}(A)) &=& 
\operatorname{tr}(A^k)\qquad\forall A\in\SLL_2
\\
\Psi_{2k}(x)- 2& =& \begin{cases} (x-2) \Phi_k(x)^2   & k \textrm{ odd}\\
                     (x^2-4)  \Phi_k(x)^2   & k \textrm{ even}
                    \end{cases}
   \label{eqn:psi2kminus2} \\
\Psi_{2k}(x)+ 2& = &    \begin{cases} (x+2) \Psi_k(x)^2  & k \textrm{ odd}\\
                          \Psi_k(x)^2   & k \textrm{ even}
                    \end{cases} 
                    \label{eqn:psi2kplus2}
\end{eqnarray}
The proofs of these formulae and other properties of $\Phi_k$ and $\Psi_k$ are
provided in Appendix~\ref{section:PhiPsi}.

\medskip

We describe now $X(T(m,n))$. Using coordinates
$$
\begin{array}{l}
x=\operatorname{tr}_a \\
y=\operatorname{tr}_b \\
z=\operatorname{tr}_{ab^{-1} }
\end{array}
$$
we have
$$
X(T(2,n))= X^{ab}(T(2,n))\cup \{(x,y,z)\in\CC^3\mid x=\Psi_n(y)=0\},
$$
and, for $m>2$,
\begin{multline*}
X(T(m,n))= X^{ab}(T(m,n))\cup \{(x,y,z)\in\CC^3\mid \Phi_m(x)=\Phi_n(y)=0\}
 \\
 \cup \{(x,y,z)\in\CC^3\mid \Psi_m(x)=\Psi_n(y)=0\}
\end{multline*}
In both cases there are $(m-1)(n-1)/2+1$ components, including $X^{ab}(T(m,n))$.

The components other than  $X^{ab}(T(m,n))$ are lines in the coordinates
$(x,y,z)$, as they are defined by $x=x_0$ and $y=y_0$,
for some values $x_0,y_0\in\CC$.

To understand ramifications we use the following lemma, which is
straightforward keeping in mind that the map sending each element of a field of
characteristic $p$ to its $p$th power is a morphisms of the field:
\begin{Lemma}
\label{lemma:phipsi}
For a prime $p>2$, if $p\mid k$ and $k=p^{r}k'$ with $r$ maximal we have:
$$
\Phi_k(u)\equiv\begin{cases}
 \Phi_{k'}(u)^{p^r} (u-2)^{(p^r-1)/2}  & \textrm{for } k\textrm{ odd,}  \\[7pt]
 \Phi_{k'}(u)^{p^r} (u^2-4)^{(p^r-1)/2}   &\textrm{for } k\textrm{ even,} 
                                \end{cases} \mod p,
$$
and
$$
\Psi_k(u)\equiv\begin{cases}
  \Psi_{k'}(u)^{p^r} (u+2)^{(p^r-1)/2}  & \textrm{for }k\textrm{ odd.}  \\[7pt]
  \Psi_{k'}(u)^{p^r}    & \textrm{for }k\textrm{ even,} 
                                \end{cases} \mod p.
$$
\end{Lemma}

In the statement of the lemma, notice that $\Phi_1=\Psi_1=1$.

Notice that in characteristic $p$ an element has trace $\pm 2$ precisely when
its $p$-th power is equal to $\pm \operatorname{Id}$.
In our situation, when $x$ or $y $ is $\pm 2\mod p$, the representation can
still be irreducible. Thus:

\begin{Corollary}
Let $\FF$ be an algebraically closed field of characteristic $p>2$.
If $p\mid n$,  then the $(m-1)(n-1)/2$ components of $X(T(m,n))$ containing
irreducible characters collapse to $(m-1)(n'+1)/2$ components in
$X(T(m,n))_\FF$,
all of them containing irreducible characters, where $n'$ is coprime with $p$ 
satisfying  $n=p^r n'$.
\end{Corollary}

More specifically, among the $(m-1)(n-1)/2$ components present in
characteristic zero, $(m-1)(n-p^r)/2$ collapse to $(m-1)(n'-1)/2$ components
in groups of $p^r$, and $(m-1)(p^r-1)/2$ collapse to $(m-1)/2$ components
(that contain irreducible characters, as $z$ may take arbitrary values).

The components with irreducible characters are pairwise disjoint, but we may
ask how they intersect $X^{ab}(T(m,n))$.

\begin{Lemma}
In characteristic 0 or $p$ coprime with $m$ and $n$, $X^{ab}(T(m,n))$
intersects each other irreducible component of $X(T(m,n))$ transversely in two
distinct points.

When reducing mod $p$, if $p\mid n$ then $X^{ab}(T(m,n))$ intersects the
components with $y=\pm 2$ tangentially at single point.
\end{Lemma}

\begin{proof}
A necessary and sufficient condition for a representation to be abelian or
reducible is that the trace of the commutator of the generators $a$ and $b$ is
2; namely, that the coordinates satisfy
$$
x^2+y^2+z^2-xyz-4=0.
$$
The discriminant in $z$ of this equation is $(x-2)(x+2)(y-2)(y+2)$. This
implies that the intersection with a line given by $x=x_0$ and $y=y_0$
is transverse and consists precisely of two points if and only if
$x_0\neq \pm 2$ and $y_0\neq \pm 2$. This holds in characteristic 0 or
coprime with $m$ and $n$.

When $p\mid n$, if $y_0\neq \pm2$ then the discriminant does not vanish and the
previous discussion applies. So assume $y_0=2$. Notice that the abelianisation
$\pi_1(\mathbf{S}^3\setminus T(m,n))\to \ZZ$ maps $a$ to $n$ and $b$ to $m$. 
Hence, for any element $\gamma\in \pi_1(\mathbf{S}^3\setminus T(m,n))$ whose 
abelianisation generates $\ZZ$, if $\tau=\operatorname{tr}_\gamma$, then by 
\eqref{eqn:psi2ktraceAk} $X^{ab}(T(m,n))$ is parametrised by:
$$
x= \operatorname{tr}_{\gamma^n }= \Psi_{2n}(\tau), \quad 
y= \operatorname{tr}_{\gamma^m }= \Psi_{2m}(\tau), \quad
z= \operatorname{tr}_{\gamma^{-m+n}}= \Psi_{2|m-n|}(\tau).
$$
We are going to show that the derivatives of $x$ and $y$ with respect to $\tau$
vanish at any intersection point with coordinates $(x_0,2,z_0)$ (this yields
tangency with the component defined by $x=x_0$ and $y=y_0$). Since $p\mid n$,
by Lemma~\ref{lemma:phipsi} $x=\Psi_{2n}(\tau)\cong (\Psi_{2n}(\tau))^p$ and
its derivative $\mod p$  with respect to $\tau$ vanishes for any $\tau$.
For the derivative of $y$, write $y= 2+ (\Psi_{2m}(\tau)- 2)$ and use
\eqref{eqn:psi2kminus2}:
$$
y(\tau)= 2+ \begin{cases} 
             (\tau-2) \Phi_m(\tau)^2 & \textrm{for }k \textrm{ odd,} \\
             (\tau^2-4) \Phi_m(\tau)^2 & \textrm{for }k \textrm{ even.}
            \end{cases}
$$
As $x_0\neq \pm 2\mod p$, then $\tau_0\neq \pm 2$, hence $\Phi_m(\tau_0)=0$ and
the derivative of $y$ with respect to $\tau$ at $\tau=\tau_0$ vanishes.

When $y=-2$, the proof is similar, using \eqref{eqn:psi2kplus2}.
\end{proof}

\section{Satellite knots}
\label{sec:sat}

We will study in detail two examples, both with companion the trefoil knot.
Recall from the section on torus knots that the fundamental group of the
trefoil knot $3_1$ is
$$\pi_1(\mathbf{S}^3\setminus T(2,3))\cong 
\langle \beta,\gamma \mid \beta^2=\gamma^3\rangle,$$
where the fibre of the Seifert fibration is $\beta^2=\gamma^3$ and
$\gamma^{-1}\beta$ represents a meridian.

From this presentation, the character variety of the trefoil has coordinates
$
(\operatorname{tr}_\beta,  \operatorname{tr}_\gamma,  \operatorname{tr}_{\gamma^{-1}\beta})
$.
It has two components. Both components correspond to characters of
representations of quotients of the group: the characters of the first 
component correspond to representations where the centre is sent to minus the 
identity
$$
X^{irr}(3_1 )=\{(\operatorname{tr}_\beta,  \operatorname{tr}_\gamma,  \operatorname{tr}_{\gamma^{-1}\beta})\in\CC^3  \mid \operatorname{tr}_\beta=0,  \ \operatorname{tr}_\gamma=1 \},
$$
the characters of the second component correspond to abelian representations
$$
X^{ab}(3_1 )=\{(\operatorname{tr}_\beta,  \operatorname{tr}_\gamma,  \operatorname{tr}_{\gamma^{-1}\beta})\in\CC^3  \mid
\operatorname{tr}_\beta= \operatorname{tr}_{\gamma^{-1}\beta}^3-3 \, \operatorname{tr}_{\gamma^{-1}\beta},
\ 
\operatorname{tr}_\gamma= \operatorname{tr}_{\gamma^{-1}\beta}^2-2\},
$$
as ${\gamma^{-1}\beta}$ represents a meridian.
Because every representation of the satellite knots we will be considering 
induces a representation of the trefoil knot, one way to determine the 
representations, and hence the characters, of the satellite knots is to 
establish which representations of the patterns (in our case the $T(2,4)$ torus 
link, and the Whitehead link) can be glued along the common boundary to 
representations of the trefoil knot to give global representations.

\begin{center}
  \textsc{Example: a cable knot $K_c$}
 \end{center}

This knot exterior is obtained by gluing the exterior of the trefoil knot to
the exterior of the $T(2,4)$-torus link in such a way that the fibre of the
trefoil knot is identified to the meridian of one of the components of $T(2,4)$
while the meridian of the trefoil knot is glued to the longitude: this ensures
that the resulting knot has tunnel number one. A presentation for the torus
link is
$$\pi_1(\mathbf{S}^3\setminus T(2,4))\cong \langle \alpha, b \mid 
b\alpha b^{-1}\alpha=\alpha b^{-1}\alpha b \rangle. $$
Here $\alpha$ and $b$ represent meridians of the two components of $T(2,4)$.
One easily observes that the longitude associated to $b$ is
$\alpha b^{-1}\alpha b$, so that the exterior of the cable knot has the
following presentation:
$$\pi_1(\mathbf{S}^3\setminus K_c)\cong \langle \alpha, \beta \mid
b=\beta^2, \gamma^{-1}\beta=\alpha b^{-1}\alpha b, \beta^2=\gamma^3, 
b\alpha b^{-1}\alpha=\alpha b^{-1}\alpha b \rangle. $$

The variety of characters of $T(2,4)$ has coordinates
$(\operatorname{tr}_\alpha, \operatorname{tr}_b,
\operatorname{tr}_{\alpha b^{-1}}) $ and two components. One that maps the 
centre to the identity:
$$
X^{irr}(T(2,4))=\{  (\operatorname{tr}_\alpha, \operatorname{tr}_b,\operatorname{tr}_{\alpha b^{-1}}) \in\CC^3\mid  \operatorname{tr}_{\alpha b^{-1}}=0   \}
$$
and the abelian component
$$
X^{ab}(T(2,4))=\{  (\operatorname{tr}_\alpha, \operatorname{tr}_b,
\operatorname{tr}_{\alpha b^{-1}}) \in\CC^3\mid 
\operatorname{tr}_\alpha^2+ \operatorname{tr}_b^2+
\operatorname{tr}_{\alpha b^{-1}}^2 - 
\operatorname{tr}_\alpha \operatorname{tr}_b \operatorname{tr}_{\alpha b^{-1}}
-4=0 \}.
$$

To compute $X(K_c)$ we use coordinates
$$
t=\operatorname{tr}_\alpha,\ {z}=\operatorname{tr}_\beta,\  
x=\operatorname{tr}_{\alpha\beta^{-1}} .
$$
We start with $X(3_1)$ and $X(T(2,4))$ and we add the relations:
\begin{equation}
\label{eqn:cable}
 \begin{array}{l}
  \operatorname{tr}_b = \operatorname{tr}_{\beta^2}= {z}^2-2\\
 \operatorname{tr}_{\alpha b^{-1}}  =  \operatorname{tr}_{\alpha\beta ^{-2}}  = 
{z}\, x-t \\
 \operatorname{tr}_{\gamma}  =  
\operatorname{tr}_{\beta^{-1}\alpha^{-1} \beta^{2}\alpha^{-1} \ }  = 
t{{z}}^{2}x-{t}^{2}{z}-{{z}}^{3}-{z}{x}^{2}+tx+3\,{z} \\
 \operatorname{tr}_{\beta\gamma^{-1}}  =  
\operatorname{tr}_{\alpha b^{-1} \alpha b}= 
t{{z}}^{3}x-{t}^{2}{{z}}^{2}-{{z}}^{4}-{{z}}^{2}{x}^{2}+{t}^{2}+4\,{{z}}^{2}-2
 \end{array}
\end{equation}
We assume first that the restriction to $\pi_1(\mathbf{S}^3\setminus 3_1)$ lies 
in $X^{irr}( 3_1 )$. This imposes the conditions ${z}=0$ and 
$\operatorname{tr}_\gamma=1$, which in \eqref{eqn:cable} yield
$$
 \operatorname{tr}_b=-2, \quad  \operatorname{tr}_{\alpha b^{-1}}=-t, \quad 
1=t\, x, \quad \operatorname{tr}_{\beta\gamma^{-1}}= t^2-2
$$
These equalities are incompatible with $t_{\alpha b^{-1}}=0$, hence the 
intersection with $X^{irr}(T(2,4))$ is empty. On the other hand, these 
equalities imply that  
$$\operatorname{tr}_\alpha^2+ \operatorname{tr}_b^2+
\operatorname{tr}_{\alpha b^{-1}}^2 - 
\operatorname{tr}_\alpha \operatorname{tr}_b \operatorname{tr}_{\alpha b^{-1}}
-4=0,
$$
that defines $X^{ab}(T(2,4))$. Thus we get a first component
$$
X_1(K_c)=\{(t,{z},x)\in\CC^3\mid {z}=0,\ t\, x=1\}.
$$
Next we suppose that the restriction to $\pi_1(\mathbf{S}^3\setminus 3_1)$ lies 
in $X^{ab}( 3_1 )$, namely it is abelian. Therefore the corresponding 
representations factor through a Dehn filling on $T(2,4)$. One way to compute 
the Dehn filling is to add a generator $\eta$ and the relations
$$
\beta=\eta^3,\qquad \gamma=\eta^2,
$$
as those are the relations that abelianize the trefoil. The quotient of 
$\pi_1(\mathbf{S}^3\setminus K_c)$ by these relations is
$$
\pi_1(\mathbf{S}^3\setminus K_c)/\langle \beta=\eta^3,  \gamma=\eta^2\rangle\cong \langle\alpha,\eta\mid (\alpha^{-1}\eta^6)^2=\eta^{11}\rangle.
$$
Thus we get the group of the torus knot $T(2,11)$, and therefore
$$
X(K_c)\cong X_1(K_c)\cup X(T(2,11)).
$$
The six components of $X(T(2,11))$ in the coordinates $(t,{z},x)$ become:
$$
X_2(K_c)=\{(t,{z},x)\in\CC^3\mid  
  t=x\,{z}, \quad   
      {{z}}^{5}-{{z}}^{4}-4\,{{z}}^{3}+3\,{{z}}^{2}+3\,{z}
-1 =0
  \},
$$
for the five curves of irreducible representations, and 
$$
X^{ab}(K_c)= \{ (t,{z},x)\in\CC^3\mid x= {t}^{5}-5\,{t}^{3}+5\,t, \ 
{z}= {t}^{6}-6\,{t}^{4}+9\,{t}^{2}-2  \}.
$$
This last component is the abelian one. Notice that the polynomial equation 
says that  $x$ equals the trace of $\alpha^5$ and ${z}$ that of $\alpha^6$.
Also notice that the five components in $X_2(K_c)$ collapse to a single one 
mod 11, because
$$
  \left( {{z}}^{5}-{{z}}^{4}-4\,{{z}}^{3}+3\,{{z}}^{2}+3\,{z}
-1 \right) \equiv ({z}-2)^5\mod 11 .
$$


%

\begin{center}
  \textsc{Example: a Whitehead double $K_w$}
 \end{center}

In this example, the cable space $S^3-T(2,4)$ is replaced with the exterior of
the Whitehead link, with presentation
$$
\pi_1(\mathbf{S}^3\setminus 5_1^2)=
\langle \alpha, b \mid \alpha b^{-1}\alpha^{-1}b\alpha^{-1}b^{-1}\alpha=
b^{-1}\alpha b^{-1}\alpha^{-1}b \alpha^{-1}b^{-1}\alpha b\rangle ,
$$
where $\alpha$ and $b$ are meridians, see Figure~\ref{fig:DehnFilling}.
The gluing is as in the previous example, thus giving the following presentation
\begin{multline*}
\pi_1(\mathbf{S}^3\setminus K_w)\cong \langle \alpha, \beta \mid
b=\beta^2, 
\gamma^{-1}\beta=\alpha b^{-1}\alpha^{-1}b\alpha^{-1}b^{-1}\alpha b, 
\beta^2=\gamma^3,
\\
\alpha b^{-1}\alpha^{-1}b\alpha^{-1}b^{-1}\alpha=
b^{-1}\alpha b^{-1}\alpha^{-1}b \alpha^{-1}b^{-1}\alpha b
\rangle.
\end{multline*}

\begin{Claim}
There is no character in $X(K_w)$ whose restriction to the trefoil lies in 
$X^{irr}(3_1)$.
\end{Claim}

\begin{proof}
Assume there is such a character $\chi_\rho$. The gluing requires that a 
meridian of the Whitehead link, $b$, is glued to the fibre of the trefoil. 
Hence, $\rho(b)=-\operatorname{Id}$, i.e.~it is trivial in 
$\mathrm{PSL}_2(\CC)$. We notice that the image by $\rho$ of the longitude that 
commutes with $b$ must also be trivial in $\mathrm{PSL}_2(\CC)$, because each 
component of the Whitehead link is a trivial knot and the linking number 
between the two components vanishes. As there is no representation in 
$X^{irr}(3_1)$ whose restriction to the peripheral subgroup is trivial in 
$\mathrm{PSL}_2(\CC)$, we get a contradiction.
\end{proof}

Hence all characters in $X(K_w)$ restrict to abelian characters in 
$X^{ab}(3_1)$. Therefore, we obtain the same representations we get by 
replacing the trefoil knot exterior by a solid torus. In other terms, these
representations correspond to representations of the knot obtained by
Dehn-filling one of the two components of the Whitehead link $(5^2_1)_r$:
$$
X(K_h)\cong X((5^2_1)_{r}),
$$
where 
$r\in\mathbb{Q}$ is the slope of the filling.
To determine $r$, 
we add the generator $\eta$ and the relations
$$
\beta=\eta^3,\qquad \gamma=\eta^2,
$$
to the presentation of $\pi_1(\mathbf{S}^3\setminus K_w)$
as in the previous example, because those are precisely the abelianisation 
relations of the trefoil. In particular $\eta$ is a longitude that commutes 
with $b$ (i.e.~they are in the same component and in the same choice of 
peripheral group in the conjugacy class). Thus, we get the presentation
$$
\pi_1(\mathbf{S}^3\setminus (5^2_1)_r)=
\langle \alpha, b, \gamma \mid
b= \gamma^6, 
\gamma=\alpha b^{-1}\alpha^{-1}b\alpha^{-1}b^{-1}\alpha b, 
b \gamma= \gamma b\rangle .
$$
Hence, the filling meridian is $b \gamma^{-6}$, i.e.~the filling slope is 
$r=-1/6$.

\begin{figure}[h]
 \begin{center}
\begin{tikzpicture}[scale=.5,line cap=round]
\draw [thick, color=black] (2,0) arc [radius=2, start angle=180, end angle= 0];
\draw [thick, color=black] (0,0) arc [radius=4, start angle=180, end angle= 0];
\draw [thick, color=black] (0,0) arc [radius=1, start angle=180, end angle= 210];
\draw [thick, color=black] (2,0) arc [radius=1, start angle=360, end angle= 240];
\draw [thick, color=black] (0,-1.5) arc [radius=4, start angle=180, end angle= 360];
\draw [thick, color=black] (2,-1.5) arc [radius=2, start angle=180, end angle= 360];
\draw [thick, color=black] (0,-1.5) arc [radius=1, start angle=180, end angle= 60];
\draw [thick, color=black] (2,-1.5) arc [radius=1, start angle=0, end angle= 30];
\draw[thick, color=black]  (5.2,.5) arc (135:415: 2.5 and .8);
\draw [color=blue, thick] (6.7,3.2)--(7.1,3.6);  
\draw [->, color=blue, thick] (6.4,2.9)--(5.9,2.4); 
\node [color=blue] at (6.6,3.9) {$a$};
 \draw [line width=5pt, color=white] (4.5,-.75) arc [radius=.5, start angle=235, end angle= 45]; 
 \draw [thick, color=blue, ->] (4.5,-.75) arc [radius=.5, start angle=235, end angle= 45]; 
\node [color=blue] at (3.9,0) {$b$};
\draw[thick, color=magenta, ->]  (6,-1) arc (245:300: 2.5 and .8);
\node [color=magenta] at (7.5,-1.5) {$\gamma$};
 \end{tikzpicture}
 \end{center}
  \caption{The meridians $a$ and $b$ and the longitude $u$.} 
 \label{fig:DehnFilling}
\end{figure}
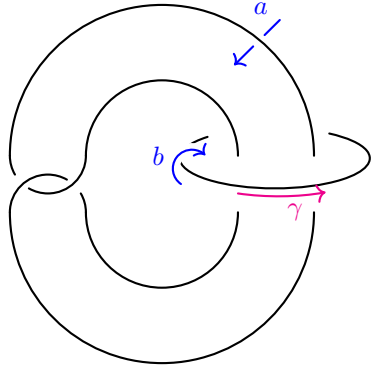

To understand the manifold from the Dehn filling with slope $b \gamma^{-6}$,
the standard procedure is to cut open along a disc bounded by the component of 
the link with meridian $b$, apply 6 turns, glue back again, so that the new 
filling slope becomes $1/0=\infty$. This leads to a twist knot with twelve 
half-twists, that would not be alternate, and simplifies to 11 half-twists as 
in Figure~\ref{fig:1123}. This is the two-bridge knot $11/23$.
The variety of characters of this two-bridge knot has been discussed in 
Section~\ref{sec:twobridge}. It ramifies for $p=23$.

\color{black}

\begin{figure}[h]
 \begin{center}
  \begin{tikzpicture}[scale=.5,line cap=round]
\begin{scope}[yscale=.75]
    \braid[line width=1pt,  color=black] s_1 s_1 s_1 s_1 s_1 s_1 s_1 s_1 s_1 s_1 s_1  ;
\end{scope}
\draw [thick, color=black] (1,0) arc [radius=.25, start angle=0, end angle= 90];
\draw [thick, color=black] (2,0) arc [radius=1.25, start angle=0, end angle= 90];
\draw [thick, color=black] (1,-8.625) arc [radius=.25, start angle=0, end angle= -90];
\draw [thick, color=black] (2,-8.625) arc [radius=1.25, start angle=0, end angle= -90];
\draw [thick, color=black] (0.75,0.25) arc [radius=4, start angle=90, end angle= 180];
\draw [thick, color=black] (0.75,1.25) arc [radius=5, start angle=90, end angle= 180];
\draw [thick, color=black] (0.75,-8.875) arc [radius=4, start angle=-90, end angle= -180];
\draw [thick, color=black] (0.75,-9.875) arc [radius=5, start angle=-90, end angle= -180];
\draw [thick, color=black] (-4.25,-3.75) arc [radius=.5, start angle=180, end angle= 290];
\draw [thick, color=black] (-3.25,-3.75) arc [radius=.5, start angle=360, end angle= 330];
\draw [thick, color=black] (-4.25,-4.85)--(-4.25,-4.5);
\draw [thick, color=black] (-3.25,-4.85)--(-3.25,-4.5);
\draw [thick, color=black] (-4.25,-4.5) arc [radius=.5, start angle=180, end angle= 150];
\draw [thick, color=black] (-3.25,-4.5) arc [radius=.5, start angle=0, end angle= 120];
\draw [thick, decorate,decoration={brace,amplitude=10pt,mirror,raise=4pt},yshift=0pt]
(2.25,-8.5) -- (2.25,-.5) node [black,midway,xshift=0.8cm] {\footnotesize
$11$};

\begin{scope}[shift={(-12,.5)}]
 \begin{scope}[yscale=.75]
    \braid[line width=1pt,  color=black] s_1 s_1 s_1 s_1 s_1 s_1 s_1 s_1 s_1 s_1 s_1 s_1 ;
\end{scope}
\draw [thick, color=black] (1,0) arc [radius=.25, start angle=0, end angle= 90];
\draw [thick, color=black] (2,0) arc [radius=1.25, start angle=0, end angle= 90];
\draw [thick, color=black] (1,-9.325) arc [radius=.25, start angle=0, end angle= -90];
\draw [thick, color=black] (2,-9.325) arc [radius=1.25, start angle=0, end angle= -90];
\draw [thick, color=black] (0.75,0.25) arc [radius=4, start angle=90, end angle= 180];
\draw [thick, color=black] (0.75,1.25) arc [radius=5, start angle=90, end angle= 180];
\draw [thick, color=black] (0.75,-9.575) arc [radius=4, start angle=-90, end angle= -180];
\draw [thick, color=black] (0.75,-10.575) arc [radius=5, start angle=-90, end angle= -180];
\begin{scope}[shift={(0,-.5)}]
\draw [thick, color=black] (-4.25,-3.75) arc [radius=.5, start angle=180, end angle= 210];
\draw [thick, color=black] (-3.25,-3.75) arc [radius=.5, start angle=360, end angle= 240];
\draw [thick, color=black] (-4.25,-5.1)--(-4.25,-4.5);
\draw [thick, color=black] (-3.25,-5.1)--(-3.25,-4.5);
\draw [thick, color=black] (-4.25,-4.5) arc [radius=.5, start angle=180, end angle= 60];
\draw [thick, color=black] (-3.25,-4.5) arc [radius=.5, start angle=0, end angle= 30];
\draw [thick, color=black] (-4.25,-3.75)--(-4.25,-3.25);
\draw [thick, color=black] (-3.25,-3.75)--(-3.25,-3.25);
\end{scope}
\draw [thick, decorate,decoration={brace,amplitude=10pt,mirror,raise=4pt},yshift=0pt]
(2.25,-8.5) -- (2.25,-.5) node [black,midway,xshift=0.8cm] {\footnotesize
$12$};
\end{scope}
 \end{tikzpicture}
 \end{center}
  \caption{The projection on the left has 12 half-twists but it is not 
alternate, it simplifies to the projection on the right, which is a diagram of 
the $2$-bridge knot $11/23$.} 
 \label{fig:1123}
\end{figure}
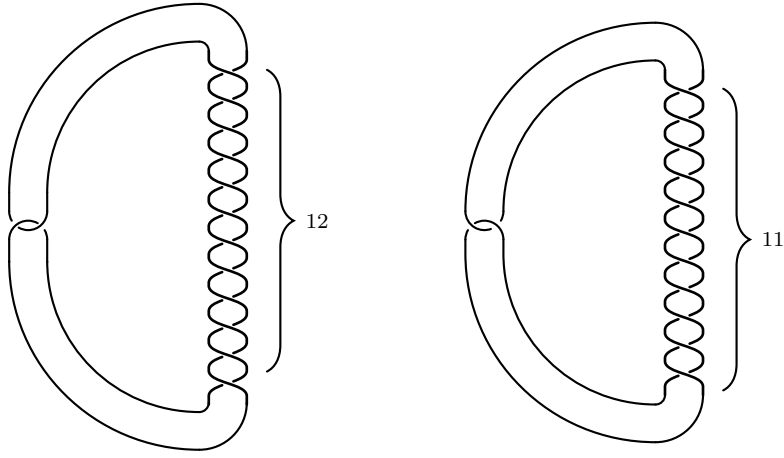



\begin{appendices}

\section{Computations for the $8_5$}
\label{appendix:8_5}

From Figure~\ref{fig:8_5}, using Wirtinger presentation we compute:
$$
\pi_1(\mathbf{S}^3\setminus 8_5)\cong \langle\alpha, \beta, \gamma \mid [\alpha^{-1},\gamma^{-1}] \alpha [\beta^{-1},\alpha]=\gamma,\
                                         [\beta^{-1},\alpha] \beta [\gamma,\beta]=\gamma
\rangle
$$
Next we perform elementary changes to obtain a presentation with two 
generators. Firstly we introduce the elements
$ w=\alpha\gamma$ and $ z=\beta\gamma
$
and we eliminate  $\alpha$ and $\beta$: 
\begin{multline*}
\pi_1(\mathbf{S}^3\setminus 8_5)\cong \langle \gamma, w, z\mid
w^{-1} \gamma^{-1} w^2 z^{-1} w \gamma^{-1} z w^{-1} \\ =
z^{-2} w \gamma^{-1} z w^{-1} z^2 \gamma^{-1} z^{-1} = 1  \rangle
 \end{multline*}
Setting $a=w^{-1} z^2$ and eliminating $w$ we get:
\begin{multline*}
 \pi_1(\mathbf{S}^3\setminus 8_5)\cong \langle \gamma, a, z\mid
 a z^{-2} \gamma^{-1} z^2 a^{-1} z^2 a^{-1} z (a^{-1} \gamma^{-1} z a) z^{-2} = \\
a^{-1} \gamma^{-1} z a \gamma^{-1}  =1
\rangle 
\end{multline*}
The second relation allows to write $z=\gamma a \gamma a^{-1}$  and to simplify the bracket in the first one:
$$
a z^{-2} \gamma^{-1} z^2 a^{-1} z^2 a^{-1} z \gamma z^{-2}=1   .
$$
From this relation we get
$$
a z^{-2}\gamma^{-1}z^2 a^{-1} = z^2\gamma^{-1} z^{-1} a z^{-2},
$$
which tells that $\gamma^{-1}z^{-1}a$ is conjugate to $\gamma^{-1}$. Taking traces it yields the equation
\begin{equation}
\label{eq:8_5_1}
 {t}^{2}{x}^{2}y+t \left( -{x}^{3}-2\,x{y}^{2}+2\,x-1 \right) +{x}^{2}y
+{y}^{3}-3\,y
=0
\end{equation}
Next taking traces on the relation
$$
\gamma^{-1} z^2 a^{-1} z^2 a^{-1} z\gamma= z^2 a z^2
$$
gives
\begin{multline}
 \label{eq:8_5_2}
  ( {t}^{2}{x}^{3}{y}^{2}-2\,t{x}^{4}y-2\,t{x}^{2}{y}^{3}-{y}^{2}x
{t}^{2}+{x}^{5}+2\,{x}^{3}{y}^{2}+x{y}^{4}+6\,t{x}^{2}y+t{y}^{3}
\\
-txy-5
\,{x}^{3}-5\,x{y}^{2}-2\,ty+{x}^{2}+{y}^{2}+5\,x-2 ) \\
\times ( {t}
^{3}{x}^{4}{y}^{3}-3\,{t}^{2}{x}^{5}{y}^{2}-3\,{t}^{2}{x}^{3}{y}^{4}-{
y}^{3}{x}^{2}{t}^{3}+3\,t{x}^{6}y+6\,t{x}^{4}{y}^{3}+3\,t{x}^{2}{y}^{5
}\\
+9\,{t}^{2}{x}^{3}{y}^{2}+2\,{t}^{2}x{y}^{4}-{x}^{7}-3\,{x}^{5}{y}^{2
}-3\,{x}^{3}{y}^{4}-x{y}^{6}
-15\,t{x}^{4}y
\\
-16\,t{x}^{2}{y}^{3}-t{y}^{5
}
-4\,{y}^{2}x{t}^{2}+7\,{x}^{5}
+14\,{x}^{3}{y}^{2}
+7\,x{y}^{4}
\\
+18\,t{x
}^{2}y+4\,t{y}^{3}-14\,{x}^{3}-14\,x{y}^{2}-3\,ty+7\,x+1 ) =0
\end{multline}

We consider the ideal of the polynomials generated by \eqref{eq:8_5_1} and 
\eqref{eq:8_5_2}. We then use symbolic software to look for its prime 
decomposition (namely, the decomposition of its radical) and we get the five 
components we described. This gives an algebraic set that contains $X(8_{5})$. 
We check that it is precisely $X(8_{5})$ exploiting the fact that each 
component intersects $t=0$ in one of the points described in 
Section~\ref{sec:pretzel}.

The equations  are invariant by the symmetry $(t,x,y)\mapsto (-t, x, -y)$, 
because $t$ and $y$ are traces of elements whose projections to 
$\mathbb Z/2 \mathbb Z$ are nontrivial, and $x$ is the trace of an element that 
projects to zero in $\mathbb Z/2 \mathbb Z$.

\section{Computations for the $8_{18}$}
\label{appendix:8_18}

To obtain a presentation of its fundamental group, we start with three 
meridians $\alpha$, $\beta$ and $\gamma$ in Figure~\ref{fig:8_18}. We apply the 
map induced by the period $\phi^{\pm 1}_*$ on these generators, 
which is equivalent to moving right and left in Figure~\ref{fig:braid}.

 \begin{figure}[h]
 \begin{center}
  \begin{tikzpicture}[scale=1]
 \begin{scope}[shift={(2,0)}]
  \draw [->] (0.1,1.1) -- (0.1,1.4); 
  \draw (0.1, 0.7)--(0.1,0.9); 
  \draw [->] (0.1,2.1) -- (0.1,2.4); 
  \draw (0.1, 1.7)--(0.1,1.9);
  \draw [->] (0.1,3.1) -- (0.1,3.4); 
  \draw (0.1, 2.7)--(0.1,2.9);
 \braid[line width=1pt,  color=black, rotate=90] s_2^{-1} s_1;
  \draw [->] (2.4,1.1) -- (2.4,1.4); 
  \draw (2.4, 0.7)--(2.4,0.9); 
  \draw [->] (2.4,2.1) -- (2.4,2.4); 
  \draw (2.4, 1.7)--(2.4,1.9);
  \draw [->] (2.4,3.1) -- (2.4,3.4); 
  \draw (2.4, 2.7)--(2.4,2.9);
  \node at (-0.2,1) {$\gamma$};
  \node at (-0.2,2) {$\beta$};
  \node at (-0.2,3) {$\alpha$};
  \node at (3.2,3) {$\alpha^{-1}\beta\alpha$};
  \node at (2.7,2) {$\gamma$};
  \node at (3.2,1) {$\gamma^{-1}\alpha\gamma$};
   \end{scope}
    \begin{scope}[shift={(-2,0)}]
  \draw [->] (0.1,1.1) -- (0.1,1.4); 
  \draw (0.1, 0.7)--(0.1,0.9); 
  \draw [->] (0.1,2.1) -- (0.1,2.4); 
  \draw (0.1, 1.7)--(0.1,1.9);
  \draw [->] (0.1,3.1) -- (0.1,3.4); 
  \draw (0.1, 2.7)--(0.1,2.9);
 \braid[line width=1pt,  color=black, rotate=90] s_2^{-1} s_1;
  \draw [->] (2.4,1.1) -- (2.4,1.4); 
  \draw (2.4, 0.7)--(2.4,0.9); 
  \draw [->] (2.4,2.1) -- (2.4,2.4); 
  \draw (2.4, 1.7)--(2.4,1.9);
  \draw [->] (2.4,3.1) -- (2.4,3.4); 
  \draw (2.4, 2.7)--(2.4,2.9);
  \node at (-0.2,1) {$\beta$};
  \node at (-1.4,2) {$\beta^{-1}\gamma\beta\alpha\beta^{-1}\gamma^{-1}\beta$};
  \node at (-0.6,3) {$\beta^{-1}\gamma\beta$};
  \node at (2.7,3) {$\alpha$};
  \node at (2.7,2) {$\beta$};
  \node at (2.7,1) {$\gamma$};
   \end{scope}
\end{tikzpicture}
 \end{center}
  \caption{The generators $\alpha$, $\beta$ and $\gamma$ and their image by $\phi_*^{\pm 1}$ }
 \label{fig:braid}
\end{figure}
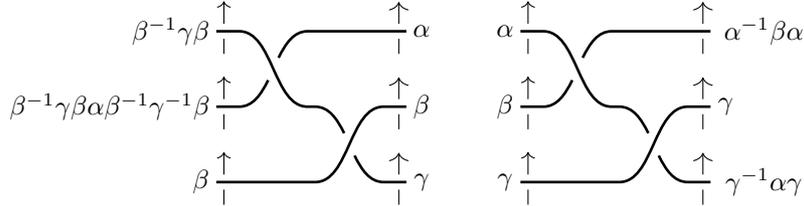
%
%
If we iterate this process once in each direction, we obtain the same elements 
in the fundamental group, yielding to relations. More precisely, the upper and 
lower strands yield
\begin{align*}
 \beta^{-1}\gamma\beta\alpha^{-1}\beta^{-1}\gamma^{-1}\beta 
\gamma\beta\alpha\beta^{-1}\gamma^{-1}\beta & = 
\alpha^{-1}\beta^{-1}\alpha\gamma\alpha^{-1}\beta\alpha \\
 \beta^{-1}\gamma\beta\alpha\beta^{-1}\gamma^{-1}\beta & = 
\gamma\alpha\gamma^{-1}\alpha^{-1}\beta\alpha\gamma\alpha^{-1}\gamma^{-1}
\end{align*}
Since one relation is redundant, we skip the relation of the middle strand, 
for it is longer. As they are written, these relations establish that the 
meridians $\alpha$, $\beta$ and $\gamma$ are conjugate, so that by taking
traces we would get the equality $t=t$. Instead, by modifying them slightly,
we derive other equivalent relations, from which we compute and equate traces:
\begin{align}
\beta\gamma\alpha\gamma^{-1}\alpha^{-1}\beta\alpha\gamma &
= \gamma\beta\alpha\beta^{-1}\gamma^{-1}\beta\gamma\alpha  \label{eq:rel1} 
\\ 
\gamma\beta\alpha^{-1}\beta^{-1}\gamma^{-1}\beta\gamma\beta\alpha\beta^{-1} & 
=  \beta\alpha^{-1}\beta^{-1}\alpha\gamma\alpha^{-1}\beta\alpha\beta^{-1}\gamma \label{eq:rel2}
\\
\gamma^{-1}\beta\gamma\alpha\gamma^{-1}\alpha^{-1}\beta\alpha\gamma & =  
\beta\alpha\beta^{-1}\gamma^{-1}\beta\gamma\alpha \label{eq:rel3}
\\
\gamma\alpha\gamma^{-1}\alpha^{-1}\beta\alpha\gamma\alpha^{-1} & 
= \beta^{-1}\gamma\beta\alpha\beta^{-1}\gamma^{-1}\beta\gamma
\label{eq:rel4}
\end{align}
Taking traces to~\eqref{eq:rel1} and writing them as polynomials in our 
coordinates yields:
\begin{multline*}
-{t}^{6}z+{t}^{6}+{t}^{4}xz+2\,{t}^{4}yz+{t}^{4}{z}^{2}-2\,{t}^{4}y-{t
}^{3}wz-2\,{t}^{2}xyz-{t}^{2}x{z}^{2}\\-{t}^{2}{y}^{2}z-{t}^{2}y{z}^{2}
-{t}^{4}+{t}^{3}w-{t}^{2}{x}^{2}+{t}^{2}xz+{t}^{2}{y}^{2}
+twxz+twyz+tw{
z}^{2}\\+x{y}^{2}z+{t}^{2}y-2\,{t}^{2}z-twy-twz-{w}^{2}z+{x}^{2}y+y{z}^{
2}-{z}^{3}+{t}^{2}-y+3\,z=0
 \end{multline*}
We do the same for~\eqref{eq:rel2}:
\begin{multline*}
 {t}^{6}x+{t}^{5}wx-{t}^{4}{x}^{2}y-{t}^{6}-{t}^{5}w+{t}^{4}{w}^{2}-{t}
^{4}{x}^{2}-{t}^{4}xy-2\,{t}^{4}xz-{t}^{3}w{x}^{2}-3\,{t}^{3}wxy
\\
-{t}^{3}wxz
+{t}^{2}{x}^{3}y
+2\,{t}^{2}{x}^{2}{y}^{2}+{t}^{2}{x}^{2}yz
+{t}^{4
}x+2\,{t}^{4}y+{t}^{4}z+3\,{t}^{3}wx+{t}^{3}wy+{t}^{2}{w}^{2}x-{t}^{2}
{w}^{2}y
\\
-{t}^{2}{x}^{2}y
+2\,{t}^{2}{x}^{2}z
+2\,{t}^{2}xyz+{t}^{2}x{z}^
{2}-tw{x}^{2}y+tw{x}^{2}z+2\,twx{y}^{2}
-{x}^{3}yz-{x}^{2}{y}^{3}
+{t}^{
4}
\\
-{t}^{3}w-3\,{t}^{2}{w}^{2}
-{t}^{2}xz-{t}^{2}{y}^{2}-{t}^{2}yz+3\,tw
xy
-2\,twxz+2\,twyz-{x}^{2}{y}^{2}-{x}^{2}{z}^{2}
\\
-2\,x{y}^{2}z-2\,{t}^{
2}x
-{t}^{2}y-twx+twz+{x}^{2}y-y{z}^{2}-{t}^{2}+tw+xz+y+2=0
\end{multline*}
For~\eqref{eq:rel3}:
\begin{equation*}
 \left( {t}^{2}-y-z \right)  \left( {t}^{3}z-{t}^{3}-{t}^{2}w-txz+xt+
ty-zt+wx \right) =0
 \end{equation*}
And for \eqref{eq:rel4}:
\begin{multline*}
- \left( {t}^{2}-y-z \right) ( {t}^{4}z-2\,{t}^{4}-{t}^{2}xz-{t}
^{2}{z}^{2}  +2\,{t}^{2}x-twy+twz+x{y}^{2}-xyz \\ +x{z}^{2}+3\,{t}^{2}-tw+yz
-2\,x-1 ) =0
 \end{multline*}
According to \cite{AcunaMontesinos} for instance, since the presentation of the
group has three generators, there is an additional polynomial relation between
the coordinates of the character variety. 
In our situation, since $t= tr_{\alpha }= tr_{\beta}= tr_{\gamma}$ the equation
reads:
\begin{equation}
\label{eq:pq}
w^2+  ( {t}^{3}-t\, x-t\, y-t\,z   )w-{t}^{2}x-{t}^{2}y-{t}^{2}z+xyz+3\,{t}^{2}
+{x}^{2}+{y}^{2}+{z}^{2}-4= 0
\end{equation}
It turns out that these five polynomials still do not generate the ideal. To
see that, we consider again the action of the period $\phi$. The induced action 
on the characters is given by $\chi\mapsto\chi\circ\phi^{-1}_*$, where $\phi_*$ 
is the action on the fundamental group, By Figure~\ref{fig:braid} (right) 
$\phi_* ^{-1}$ is given by
$$
\begin{array}{r}
\alpha \\
\beta \\
\gamma
\end{array}
\mapsto
\begin{array}{l}
 \alpha^{-1}\beta\alpha\\
 \gamma\\
 \gamma^{-1}\alpha\gamma
\end{array}
$$
We consider the ideal generated by taking the trace in 
Equations~\eqref{eq:rel1}, \eqref{eq:rel2}, \eqref{eq:rel3} and 
\eqref{eq:rel4}, their image by $\phi_*$, and Equation~\eqref{eq:pq}, thus a 
total of 9 polynomials. We then use symbolic software to look for its prime 
decomposition (namely, the decomposition of its radical) and we get the eleven 
components  we described. This gives an algebraic set that contains 
$X(8_{18})$. We check that it is precisely $X(8_{18})$ because each component 
intersects $t=0$ in one of the points described in Section~\ref{sec:t=0}.

We now describe the action of the involutions $\sigma_1$ and $\sigma_2$ on the 
variety of characters.

From Figure~\ref{fig:8_18} the action of $\sigma_1$ on the fundamental group is
$$
\begin{array}{r}
\alpha \\
\beta \\
\gamma
\end{array}
\mapsto
\begin{array}{l}
 \alpha^{-1}\beta^{-1}\alpha\\
   \alpha^{-1}\beta^{-1}\alpha^{-1}\beta\alpha \\
 \alpha^{-1}\beta^{-1} \gamma^{-1}\beta\alpha
\end{array}
$$
Up to inner homomorphisms, it can be simplified as:
$$
\begin{array}{r}
\alpha \\
\beta \\
\gamma
\end{array}
\mapsto
\begin{array}{l}
\beta^{-1}\\
   \alpha^{-1} \\
\gamma^{-1}
\end{array}
$$
Hence $\sigma_1^*\colon X(8_{18})\to X(8_{18})$  writes as:
$$
\begin{pmatrix}
 t\\
 x\\
 y\\
 z\\
 w
\end{pmatrix}
\mapsto
\begin{pmatrix}
 t \\
 x \\
 t^2-z \\
 t^2-y \\
 t^3-t\, y-t\, z+w
 \end{pmatrix}
$$

From Figure~\ref{fig:8_18},  on the fundamental group $\sigma_2$ acts as
$$
\begin{matrix}
\alpha \\
\beta \\
\gamma
\end{matrix}
\mapsto
\begin{matrix}
 \gamma\\
 \gamma\beta\gamma^{-1}\\
 \gamma\beta\alpha\beta^{-1}\gamma^{-1}
\end{matrix}
$$
and on the variety of characters 
$$
\begin{pmatrix}
 t\\
 x\\
 y\\
 z\\
 w
\end{pmatrix}
\mapsto
\begin{pmatrix}
 t \\
 y \\
 x \\
 t^2+t\, w -x\, y-z\\
 w
 \end{pmatrix}
$$

The map induced by the period $\phi$ on the variety of characters is
$$
\begin{pmatrix}
 t\\
 x\\
 y\\
 z\\
 w
\end{pmatrix}
\mapsto
\begin{pmatrix}
 t\\
 -{t}^{4}+{t}^{2}x+{t}^{2}y+z{t}^{2}+{t}^{2}-tw-xz-y \\
 {t}^{2}-z  \\
 -{t}^{4}z+{t}^{4}+{t}^{2}xz+yz{t}^{2}+{z}^{2}{t}^{2}-{t}^{2}x-{t}^{2}y
-zwt-x{z}^{2}+tw-yz+x
  \\
 -{t}^{5}+{t}^{3}x+{t}^{3}y+z{t}^{3}+2\,{t}^{3}-w{t}^{2}-xzt-xt-2\,ty-z
\end{pmatrix}
$$
From this we may derive the description of the action of the symmetry group on 
the variety of characters given in Section~\ref{sec:pihyperbolic}.

\section{The polynomials $\Phi_k$ and $\Psi_k$}
\label{section:PhiPsi}
Here we provide the proof of some of the results stated 
in Section~\ref{sec:torus} for the polynomials $\Phi_k$ and $\Psi_k$ of 
Definition~\ref{Def:PsiPhi}. First of all we rewrite  
Definition~\ref{Def:PsiPhi} to make computations simpler. The polynomials 
are determined by the conditions 
$$
\Phi_k(\lambda+{\lambda}^{-1})=\begin{cases}
                                 \dfrac{\lambda^{k/2}-\lambda^{-k/2}}{\lambda^{1/2}-\lambda^{-1/2}}  & k\textrm{ odd}  \\[7pt]
                                 \dfrac{ \lambda^{k/2}-\lambda^{-k/2} }{\lambda-\lambda^{-1}}  & k\textrm{ even} 
                                \end{cases}
$$
and
$$
\Psi_k(\lambda+{\lambda}^{-1})=\begin{cases}
                                 \dfrac{\lambda^{k/2}+\lambda^{-k/2}}{\lambda ^{1/2}+\lambda^{-1/2}} & k\textrm{ odd}  \\[7pt]
                                 ({ \lambda^{k/2}+\lambda^{-k/2} })  & k\textrm{ even}  
                                \end{cases}
$$
for every $\lambda\in\CC^*$.
As $\Psi_{2k}(\lambda+{\lambda}^{-1})=\lambda^k+\lambda^ {-k}$, 
by viewing ${\lambda}^{\pm 1}$ as the eigenvalues of a matrix, \eqref{eqn:psi2ktraceAk} is clear.

Now we give the proof of \eqref{eqn:psi2kminus2}, but we skip the proof of 
\eqref{eqn:psi2kplus2} as it is analogous.
\begin{proof}[Proof of \eqref{eqn:psi2kminus2}]
Set $u=\lambda+\lambda^{-1}$. For every $k$ we have
\begin{equation}
 \label{eqn:Psi2k}
\Psi_{2k}(u)- 2= \lambda^k+\lambda^ {-k}-2= (\lambda^{k/2}-\lambda^{-k/2} )^2 
\end{equation}
When $k$ is odd, 
\eqref{eqn:Psi2k} gives:
\[
\Psi_{2k}(u)- 2= ( {\lambda^{1/2}-\lambda^{-1/2}}  )^2\left( \dfrac{ \lambda^{k/2}-\lambda^{-k/2}  }{\lambda^{1/2}-\lambda^{-1/2}}\right)^2=
(u-2)\Phi_{k}(u) ^2 .
\]
When $k$ is even:
\[
\Psi_{2k}(u)- 2= ( {\lambda-\lambda^{-1}}  )^2\left( \dfrac{ \lambda^{k/2}-\lambda^{-k/2}  }{\lambda-\lambda^{-1}}\right)^2=
(u^2-4)\Phi_{k}(u) ^2  , 
\]
which concludes the proof of \eqref{eqn:psi2kminus2}.
\end{proof}

\begin{proof}[Proof of Lemma~\ref{lemma:phipsi}]
Let $k=p^r k'$ be as in the statement of the lemma. The main idea is to use the 
identity
$$
\lambda^{k/2}\pm \lambda^{-k/2}\equiv  (\lambda^{k'/2}\pm\lambda^{-k'/2})^{p^r}\mod p
$$
and to adapt it to each case. For instance for $k$ odd:
\begin{multline*}
\Phi_k( u ) =\frac{ \lambda^{k/2}- \lambda^{-k/2} }{\lambda^{1/2}-\lambda^{-1/2}}
 \equiv
 \frac{  (\lambda^{k'/2}-\lambda^{-k'/2})^{p^r}  }{ (\lambda^{1/2}-\lambda^{-1/2})^{p^r} }
 { (\lambda^{1/2}-\lambda^{-1/2})^{p^r-1} }
 \\
 = \Phi_{k'}( u )^{p^r}  ( u-2)^{( p^r-1   )/2}.
\end{multline*}
A similar trick applies to $\Phi_k$ for $k$ even and to $\Psi_k$ for $k$ either 
even or odd.
\end{proof}

Other properties are $\deg(\Phi_k)=\deg(\Psi_k)=(k-1)/2$ for $k$ odd, and 
$\deg(\Phi_k)= k/2-1$ and $\deg(\Psi_k)=k/2$ for $k$ even.

These polynomials are easy to compute if we take into account the recursive 
relations, as polynomials in $\ZZ[x]$:
\[
 \Phi_k=x \Phi_{k-2}- \Phi_{k-4}
 \qquad\textrm{ and }\qquad
 \Psi_k=x \Psi_{k-2}- \Psi_{k-4},
\]
or:
\[
\begin{cases}
  \Phi_{2k}=\Phi_{2k-1} -\Phi_{2k-2} \\
  \Phi_{2k+1}=(x+2)\Phi_{2k} -\Phi_{2k-1} 
\end{cases}
\quad
\textrm{ and }\quad
\begin{cases}
 \Psi_{2k}=(x+2)\Psi_{2k-1} -\Psi_{2k-2}\\
  \Psi_{2k+1}=\Psi_{2k}-\Psi_{2k-1} 
\end{cases}
%
\]
with
$\Phi_1=\Phi_2=\Psi_1=1$  and $\Psi_2=x$. 
Here are the polynomials for  $k$ up to 10:
\begin{align*}
 \Phi_1&=1   & \Psi_1&=1 \\
 \Phi_2&=1   & \Psi_2&=x \\
  \Phi_3&=x+1   & \Psi_3&=x-1 \\
   \Phi_4&=x   & \Psi_4&=x^2-2 \\
 \Phi_5&=x^2+x-1   & \Psi_5&=x^2-x-1 \\
 \Phi_6&=x^2-1   			& \Psi_6&=x^3-3x \\
 \Phi_7&={x}^{3}+{x}^{2}-2\,x-1   	 & \Psi_7&= {x}^{3}-{x}^{2}-2\,x+1 \\
 \Phi_8&= {x}^{3}-2\,x  		& \Psi_8&= {x}^{4}-4\,{x}^{2}+2 \\
 \Phi_9&= {x}^{4}+{x}^{3}-3\,{x}^{2}-2\,x+1  & \Psi_9&= {x}^{4}-{x}^{3}-3\,{x}^{2}+2\,x+1 \\
 \Phi_{10}&= {x}^{4}-3\,{x}^{2}+1 	& \Psi_{10}&= {x}^{5}-5\,{x}^{3}+5\,x
\end{align*}

\end{appendices}


%
%
%
%
%
\begin{footnotesize}
\bibliographystyle{plain}

%
%
%
%
\end{footnotesize}

\medskip

\textsc{Aix-Marseille Univ, CNRS, Centrale Marseille, I2M, UMR 7373, 13453 Marseille, France}

\texttt{luisa.paoluzzi@univ-amu.fr}

\medskip

\textsc{Departament de Matem\` atiques, Universitat Aut\` onoma de Barcelona, 08193 Cerdanyola del Vall\` es, Spain, and BGSMath}

\texttt{porti@mat.uab.cat}

\end{document}